\documentclass[11pt]{article}

\usepackage{float} 
\usepackage{amsfonts}
\usepackage{amsmath}
\usepackage{amssymb}
\usepackage{amsbsy}
\usepackage{amsthm}
\usepackage{mathtools}
\usepackage{epsfig}
\usepackage{graphicx}
\usepackage{colordvi}
\usepackage{graphics}
\usepackage{color}
\usepackage{bm}

\usepackage{fullpage}
\newtheorem{theorem}{Theorem}[section]
\newtheorem{lemma}[theorem]{Lemma}

\newtheorem{assumption}[theorem]{Assumption}
\newtheorem{alg}[theorem]{Algorithm}
\newtheorem{proposition}[theorem]{Proposition}
\newtheorem{remark}[theorem]{Remark}
\newtheorem{corollary}[theorem]{Corollary}

\numberwithin{equation}{section}

\newcommand{\norm}[1]{\left\Vert#1\right\Vert}

\newcommand{\anorm}{\norm{\,\cdot \,}}
\newcommand{\forma}{( \,\cdot \, , \, \cdot \,)}

\newcommand{\xa}{x^\alpha}

\newcommand{\wa}{w^\alpha}
\newcommand{\grad}{\nabla}
\newcommand{\ran}{\text{Range\,}}
\newcommand{\nulll}{\text{Null\,}}
\DeclareMathOperator{\argmin}{argmin}
\newcommand{\eps}{\varepsilon}
\renewcommand{\div}{\text{div\,}}
\DeclareMathOperator{\spa}{span}

\newcommand{\nr}[1]{\ensuremath{\left\|{#1}\right\|}}

\newcommand{\goto}{\rightarrow}

\newcommand{\tforall}{~\text{ for all }~}
\newcommand{\bigo}{{\mathcal O}}

\newcommand{\f}{\frac}
\DeclareMathOperator{\dd}{\,d}
\DeclareMathOperator{\ceil}{ceil}
\newcommand{\R}{\mathbb{R}}
\newcommand{\C}{\mathbb{C}}

\newcommand{\cA}{\mathcal{A}}

\begin{document}
\title{Anderson acceleration for contractive and noncontractive operators}

\author{Sara Pollock
\thanks{Department of Mathematics,
University of Florida,
Gainesville, FL, 32611; email: s.pollock@ufl.edu; partially supported by NSF DMS 1852876
and 2011519}
\and
Leo G. Rebholz
\thanks{School of Mathematical and Statistical Sciences, Clemson University,
Clemson SC 29634; email: rebholz@clemson.edu; partially supported by NSF DMS 1522191
and 2011490}
}
\maketitle

\begin{abstract}
A one-step analysis of Anderson acceleration with general algorithmic depths 
is presented.  The resulting residual bounds within both contractive and noncontractive
settings reveal the balance between the contributions from the higher and 
lower order terms, which are both dependent on the success of the optimization 
problem solved at each step of the algorithm. 
The new residual bounds show the additional terms introduced by the
extrapolation produce terms that are of a higher order than was previously understood.
In the contractive setting, these bounds sharpen previous convergence 
and acceleration results. The bounds rely on sufficient 
linear independence of the differences between consecutive residuals, rather 
than assumptions on the boundedness of the optimization coefficients,
allowing the introduction of a theoretically sound safeguarding strategy.
Several numerical tests illustrate the analysis primarily in the noncontractive
setting, and demonstrate the use of the method,
the safeguarding strategy, and theory-based guidance on dynamic 
selection of the algorithmic depth, on a p-Laplace equation,
a nonlinear Helmholtz equation,
and the steady Navier-Stokes equations with high Reynolds number in three
spatial dimensions.
{Anderson acceleration, extrapolation, noncontractive operators}
\end{abstract}

\section{Introduction}

Anderson acceleration (AA) is an extrapolation technique which recombines a 
given number of the most recent iterates and update steps in a fixed-point iteration
to improve the convergence properties of the sequence.
The coefficients of the linear combination used in the update are recomputed at each 
iteration by the solution to an optimization problem which determines a least-length 
update step.
The technique was originally introduced in the context of integral equations
in \cite{anderson65}. It has since been used in many applications
over the last decade for various types of flow problems, for instance in
\cite{BKNR19,EPRX19,LWWY12,PRX18};
geometry optimization in \cite{PDZGQL18}; 
electronic structure computations in \cite{FaSa09};
radiation diffusion and nuclear physics in \cite{AJW17,TKSHCP15};
computing nearest correlation matrices in \cite{HS16};
molecular interaction in \cite{SM11};
and on a wide range of nonlinear problems in \cite{WaNi11}, among others.  

In terms of its analysis, AA was shown to be in the class of generalized quasi-Newton
methods in \cite{Eyert96} and \cite{FaSa09}.  In \cite{WaNi11} it was shown that 
in the linear case, the variant of the method related to Type II Broyden methods
is ``essentially equivalent'' to GMRES, while the Type I variant is essentially 
equivalent to Arnoldi. In the remainder, we restrict our attention to the (standard)
Type II variant, and consider its use on the solution of nonlinear problems.
Recently in \cite{BRS18} a nontrivial ({\em cf.} \cite{WaNi11}) mathematical connection 
between AA and classical extrapolation 
algorithms used to accelerate vector sequences, including the (Modified) Minimal 
Polynomial Extrapolation, Topological and Vector Epsilon Algorithms, and Reduced Rank 
Extrapolation algorithms was established 
(see the review paper \cite{SFS87} and the references therein for further discussion on 
the relation between these more classical methods).
Meanwhile, the first mathematical results showing local convergence of AA for contractive
nonlinear operators were developed in \cite{ToKe15} and sharpened in \cite{K18}. 
The first results to prove how AA improves the convergence rate in fixed point iterations
were written by the authors in \cite{PRX18} and \cite{EPRX19}.
The present work improves upon the results of \cite{EPRX19} by further exploiting
the relationship between the optimization coefficients and optimization gain, made 
possible by analyzing the least-squares problem as it is discussed in \cite{FaSa09}
using a QR factorization. 

This paper presents a novel one-step analysis which both sharpens and generalizes
the AA convergence theory developed for contractive operators in \cite{EPRX19}.  
The new one-step estimates hold for fixed-point iterations of contractive operators or
for zero-finding fixed-point iterations based on operators whose Jacobians 
do not degenerate. The latter are of particular importance in the numerical 
approximation of nonlinear partial differential equations (PDEs).  
The presented theory does not guarantee convergence of the sequence of iterations
for noncontractive operators unless the optimization problem is assumed to be
sufficiently successful at each iteration.  However, it succeeds at explaining the
mechanism by
which AA applied to this broad class of noncontractive fixed-point operators often 
does converge, and it provides insight into the design of more robust and efficient 
algorithms, as demonstrated in the practical guidance and in the 
numerical results.

One of the fundamental aspects of the theory that (to the knowledge of the authors) 
has not been exploited in previous investigations for general algorithmic depths, 
is the relation between the
optimization coefficients and the gain from the optimization problem, which, as 
shown here, can 
be understood through a QR factorization. For this reason, the analysis is
restricted to $\R^n$ (trivially extendable to $\C^n$), 
with the norm from the optimization problem induced by an inner product.  
While the analysis and theory extend to more 
general Hilbert space settings, this allows for a clean presentation of the central 
ideas, and it is the most interesting for the solution of systems assembled from the 
discretization of nonlinear PDEs.

The presented bounds significantly sharpen those previously developed for contractive operators 
in two important ways.
First, the dependence on the higher order terms is shown to be
$\bigo\big( \| w_{k}\| ( \| w_k \| + \| w_{k-1} \|  + \ldots + \| w_{k-m} \|) \big)$, 
improving on the $\bigo(\| w_k \|^2) + \bigo(\| w_{k-1} \|^2) 
+ \ldots + \bigo(\| w_{k-m} \|^2)$
bound proven in \cite{EPRX19}, where $w_k$ is the stage-$k$ residual. 
This analysis produces the first residual bound for AA applied to 
nonlinear problems where the most recent residual $\nr{w_k}$ can be factored out 
of the entire bound; previously, the best bounds for the higher order terms involved
only older (often larger) residuals from the history.
Second, the new estimates show that if the solution to the optimization problem
does not produce a linear combination of residuals that is strictly lesser in 
norm than the most recent residual, then there is no contribution from the higher
order terms. The results of the analysis further motivate strategies 
for choosing the AA depth adaptively or dynamically, which is shown to provide a 
significant advantage over constant depths in the numerical tests.

The remainder of the paper is structured as follows.  Section \ref{sec:prelim} states
the algorithm and presents notation that will be used throughout, and
Section \ref{sec:resi-exp} summarizes the residual expansion which is similar to that of
\cite{EPRX19}. In Section \ref{sec:m1}, the new one-step analysis is presented for
algorithmic depth $m=1$, and in Section \ref{sec:genm}, 
the one-step analysis is developed
for $m > 1$.  
In section \ref{sec:practical}, practical guidance is presented on 
dynamic algorithmic depth selection and safeguarding strategies, as motivated
by the developed theory. In Section \ref{sec:numerics}, 
numerical results are presented that both illustrate the theory 
and practical guidance, and demonstrate how
AA can be effectively used to solve a nonlinear Helmholtz equation and the 
3D steady Navier-Stokes equations with Reynolds numbers past the first Hopf bifurcation.
An appendix contains a proof of a technical lemma providing particular bounds on
the entries of the 
inverse of the upper triangular matrix found in the QR decomposition.

\section{Problem setting and preliminaries}\label{sec:prelim}

Consider seeking a fixed point of Fr\'echet differentiable 
operator $g: X \goto X$ for Hilbert space $X \subseteq \R^n$ equipped with inner 
product $\forma$ and induced norm $\anorm$, under the following conditions.
\begin{assumption}\label{assume:g}
Assume $g\in C^1(X)$ has a fixed point $x^\ast$ in $X$, 
and there are positive constants $\kappa_g$ and $\hat \kappa_g$ with 
\begin{enumerate}
\item $\nr{g'(x)z} \le \kappa_g\nr{z}$ for all $x,z\in X$. 
\item $\nr{g'(x)z - g'(y)z} \le \hat \kappa_g \nr{x-y}\nr{z}$ 
for all $x,y,z \in X$.  
\end{enumerate}
\end{assumption}
A particular case of interest is finding a zero of a function $f: X \goto X$, 
where the system of nonlinear equations $f(x) = 0$, comes from the
discretization of a nonlinear PDE. 
Then $f(x) = g(x) - x$, converts between the fixed-point and zero-finding 
problems. Under Assumption \ref{assume:g} it holds that $f$ has a zero 
$x^\ast \in X$, $f \in C^1(X)$, and  
\begin{align}\label{eqn:fpL}
\nr{f'(x)z - f'(y)z} = \nr{(g'(x) - I)z - (g'(y) - I)z} \le \hat \kappa_g \nr{x-y}\nr{z},
\tforall x,y,z \in X.  
\end{align}

The AA algorithm with depth $m$ applied to the fixed-point problem 
$g(x) = x$, reads as follows.

\begin{alg}[Anderson iteration] \label{alg:anderson}
The Anderson acceleration algorithm with depth $m \ge 0$ and damping factors 
$0 < \beta_k \le 1$ 
reads: \\ 
Step 0: Choose $x_0\in X.$\\
Step 1: Find $w_1\in X $ such that $w_1 = g(x_0)-x_0$.  
Set $x_1 = x_0 + w_1$. \\
Step $k+1$: For $k=1,2,3,\ldots$ Set $m_k = \min\{ k, m\}.$\\
\indent [a.] Find $w_{k+1} = g(x_k)-x_k$. \\
\indent [b.] Solve the minimization problem for $\{ \alpha_{j}^{k+1}\}_{k-m_k}^k$
\begin{align}\label{eqn:opt-v0}
\min_{\sum_{j=k-m_k}^{k} \alpha_j^{k+1}  = 1} 
\left\| \sum_{j=k-m_k}^{k} \alpha_j^{k+1} w_{j+1} \right\| .
\end{align}
\indent [c.] For damping factor $0 < \beta_k \le 1$, set
\begin{align}\label{eqn:update-v0}
x_{k+1} =  \sum_{j= k-m_k}^k \alpha_j^{k+1} x_{j}
 + \beta_k \sum_{j= k-m_k}^k \alpha_j^{k+1} w_{j+1}.
\end{align}
\end{alg}

Throughout the remainder, the stage-$k$ differences between iterates and 
terms are defined as 
\begin{align}\label{eqn:update-err}
e_{k}  \coloneqq x_k - x_{k-1}, \quad
w_k  \coloneqq g(x_{k-1}) - x_{k-1}. 
\end{align}

The next assumption allows a key generalization from the previous convergence analysis
frameworks of \cite{EPRX19,PRX18,ToKe15,K18}, which are specific to contractive 
fixed-point operators. As discussed below in Remark \ref{rem:m1}, it is automatically
satisfied at each iteration for contractive fixed-point operators, and may be locally 
satisfied for finding zeros of nondegenerate functions.
\begin{assumption}\label{a:key} The stage-$j$ iterates and residuals satisfy the
relationship
\begin{align}\label{eqn:m1key}
\nr{w_{j+1}-w_j} \ge \sigma \nr{e_j}.
\end{align}
\end{assumption}
\begin{remark}\label{rem:m1}
Assumption \ref{a:key} is reasonable to require as it is satisfied (not necessarily 
exhaustively) under the two following important settings.
\begin{enumerate}
\item If $g$ is a contractive operator then its Lipschitz constant given by
Assumption \ref{assume:g} satisfies $\kappa_g < 1$, and by the triangle inequality
$\nr{w_{j+1} - w_j} 
\ge \nr{x_j - x_{j-1}} - \nr{g(x_j) - g(x_{j-1})} 
\ge (1-\kappa_g) \nr{e_j}.$
Then \eqref{eqn:m1key} is always satisfied with $\sigma = (1-\kappa_g)$.  
\item
In terms of seeking a zero of a function $f$ as the fixed point
of $g(x) = f(x) + x$, the nonlinear residual is $w_{j+1} = g(x_j) - x_j = f(x_j)$.
Assumption \ref{a:key} is then satisfied locally if the smallest singular value of the 
Jacobian $f'$ is uniformly bounded away from zero on $X$, and $\nr{x_{j}-x_{j-1}}$ is
small enough. Specifically, if for each 
$x,y \in X$ it holds that $\nr{f'(x)y} \ge \sigma_f\nr{y},$ for some $\sigma_f > 0$.  
This is similar to the usual assumption for Newton iterations that the Jacobian is
nondegenerate at a solution, and could be localized to the vicinity of a solution 
without undue complication.
Then, under Assumptions \ref{assume:g}, and in accordance with
\eqref{eqn:fpL}, it holds that 
\begin{align*}
\nr{f(x) - f(y)} &=  \nr{f'(y)(x-y) + \int_0^1 (f'(y + t(x-y)) - f'(y))(x-y) \dd t}
\\
& \ge \sigma_f\nr{x-y} - \f{\kappa_g}{2}\nr{x-y}^2
\\
& \ge \f{\sigma_f}{2}\nr{x-y}, ~\text{ for } \nr{x-y} \le \f{\sigma_f}{\hat \kappa_g}.
\end{align*}
Then for $\nr{e_j} \le {\sigma_f}/{\hat \kappa_g}$ it holds that
$\nr{w_{j+1} - w_j} \ge \f{\sigma_f}{2}\nr{e_j}$,
which satisfies \eqref{eqn:m1key} with $\sigma = \sigma_f/2$.  
\end{enumerate}
\end{remark}

Define the following averages given by the solution 
$\alpha^{k+1} = \{\alpha_j^{k+1}\}_{j = k-m_k}^k$ 
to the optimization problem
\eqref{eqn:opt-v0} by 
\begin{align}\label{eqn:averages}
x_k^\alpha = \sum_{j = k-m_k}^k \alpha^{k+1}_j x_j, \quad
	w_{k+1}^\alpha = \sum_{j = k-m_k}^k \alpha^{k+1}_j w_{j+1}.
\end{align}
Then the update \eqref{eqn:update-v0} can be written in terms of the averages
$\xa_k$ and $\wa_{k+1}$, by
\begin{align}\label{eqn:update-v1}
x_{k+1} = \xa_k + \beta_k \wa_{k+1}.
\end{align}
The stage-$k$ gain $\theta_k$ which quantifies the success of the optimization
problem is defined by
\begin{align}\label{eqn:thetak}
\nr{w_{k}^\alpha} = \theta_k \nr{w_{k}}. 
\end{align}
This important quantity is shown in \cite{EPRX19} to scale the first-order term in the 
residual expansion (also shown below).
Up to that scaling, this term is the residual in the standard fixed-point iteration.  
The higher-order terms on the other hand are shown below to be scaled by a 
factor of $\sqrt{1 - \theta_k^2}$, meaning a successful optimization increases the
relative weight of the higher-order terms, 
and an unsuccessful optimization increases the
relative weight of the first-order term in the residual expansion.

The constrained optimization problem \eqref{eqn:opt-v0} is often useful for
analysis of the method (see, {\em e.g.,} \cite{EPRX19, K18, PRX18, ToKe15}).  
In the current view however the following unconstrained form of the optimization 
problem \eqref{eqn:opt-v0} which is more
easily implemented in practice is also more convenient for the analysis.

Define the matrices $E_k$ and $F_{k}$ formed by the respective
differences between consecutive iterates and residuals by
\begin{align}\label{eqn:EkFkdef}
E_k &\coloneqq \left(\begin{array}{cccccc} 
e_k & e_{k-1}& \cdots & e_{k-m_k+1}
\end{array}\right),
\nonumber \\
F_k &\coloneqq \left(\begin{array}{cccccc} 
(w_{k+1}- w_k)& (w_k - w_{k-1}) & \ldots & (w_{k-m_k+2}-w_{k-m_k+1})
\end{array}\right).
\end{align}
Then \eqref{eqn:opt-v0} is equivalent to the unconstrained minimization problem
\begin{align}\label{eqn:opt-v1}
\gamma^{k+1}= \argmin_{\gamma\in \R^m} \nr{w_{k+1}- F_k \gamma}, ~\text{ for } 
\gamma^{k+1} = (\gamma^{k+1}_k, \gamma^{k+1}_{k-1},\ldots,\gamma^{k+1}_{k-m_k+1})^\top.
\end{align}
The averages $\xa_k$ and $\wa_{k+1}$ used in the update \eqref{eqn:update-v1},
and the transformation between the two sets of optimization coefficients are
related by 
\begin{align}\label{eqn:gammadef}
\xa_k = x_{k} - E_k \gamma^{k+1}, \quad
\wa_{k+1} = w_{k+1} - F_k \gamma^{k+1}, \quad
\gamma_j^{k+1} = \sum_{n = k-m_k}^{j-1} \alpha^{k+1}_n.
\end{align} 
This form of the optimization problem is instrumental in the analysis of
\cite{EPRX19}, and its direct use in the practical implementation of Algorithm 
\ref{alg:anderson} is carefully discussed in \cite{FaSa09,WaNi11}.

As commonly understood, the algorithm in its most general form 
does not identify the norm that should be used in the optimization.
The minimization problem is usually taken in the $l_2$ 
(or weighted $l_2$) sense, whereby the least-squares problem can be solved efficiently 
by a (fast) QR method (see \cite{ToKe15} for a discussion on minimizing in $l_1$ or
$l_\infty$).  Throughout the rest of this manuscript, the optimization problem 
\eqref{eqn:opt-v1} is considered the norm $\anorm$ induced by inner product $\forma$, 
which then falls under the least-squares setting. 
For example in \cite{PRX18}, the optimization is done in the $H_0^1$ sense
as the nonlinear operator there is contractive in $H_0^1$; this 
is interpreted (and implemented) as a least-squares optimization 
of a (discrete) gradient.  

The QR decomposition of $F_k$ 
will be explicitly used in the analysis to extract relations
between the optimization gain $\theta_{k}$ and optimization coefficients 
$\gamma^{k}$. A key repercussion of this  approach is that assumptions on the 
boundedness of the optimization coefficients as used in 
\cite{EPRX19,PRX18}, and \cite{K18,ToKe15} for $m > 1$, are replaced 
by assumptions on the sufficient linear independence between columns of $F_k$.  
As discussed in Subsection \ref{sec:practical}, satisfaction of these
assumptions can be easily verified and even enforced during the course of a numerical
simulation.
\section{Expansion of the residual}\label{sec:resi-exp} 
This section is summarized from \cite{EPRX19} and included here both to
make the paper more self-contained and to introduce a consistent notation.  
The novelty in the current paper is how the differences between consecutive 
iterates are bounded in terms of the nonlinear residuals under more general
assumptions than contractiveness of the underlying fixed-point operator; and,
without explicit assumptions on the boundedness of the optimization coefficients. 
The results of Sections \ref{sec:m1} and \ref{sec:genm} are applied to the residual 
expansion of this section to obtain the main results.

Starting with the definition of the residual by \eqref{eqn:update-err} and 
expanding the iterate $x_k$ by the update \eqref{eqn:update-v1}, the 
nonlinear residual $w_{k+1}$ can be written as 
\begin{align}\label{eqn:wk001}
w_{k+1} = g(x_k) - x_k 
 = (g(x_k) - \xa_{k-1}) - \beta_{k-1}\wa_k. 
\end{align} 
The first term on the right-hand side of \eqref{eqn:wk001} can be expanded by 
\eqref{eqn:averages}.
Consistent with \eqref{eqn:gammadef}, the optimization coefficients 
$\alpha_j^k$ are collected into the coefficients $\gamma_j^k$ by
$\gamma_{j}^k \coloneqq \sum_{n = k-m_{k-1}-1}^{j-1} \alpha^k_n.$ Then
\begin{align}\label{eqn:wk002}
g(x_k)& - \xa_{k-1}
 = \sum_{j = k-m_{k-1}-1}^{k-1} \alpha^k_j (g(x_k) - x_j)
\nonumber \\
& = \!\! \sum_{j = k-m_{k-1}-1}^{k-1} \alpha^k_j (g(x_j) - x_j)
+ \!\!\!\! \sum_{j = k-m_{k-1}}^{k} \left( \sum_{n = k-m_{k-1}-1}^{j-1}\alpha^k_n\right) 
(g(x_j) - g(x_{j-1}))
\nonumber \\
& = w_k^\alpha  
+ \sum_{j = k-m_{k-1}}^{k} \gamma^k_j (g(x_j) - g(x_{j-1})). 
\end{align}
This equality shows the approximation to the fixed-point $g(x_k)$ is decomposed into 
the average of the previous iterates $\xa_{k-1}$, the average over previous 
updates $\wa_k$ corresponding to the optimization
problem from the last step, and a weighted sum over the differences of consecutive
approximations.
Due to Assumption \ref{assume:g}, each term $g(x_j) - g(x_{j-1})$ has a Taylor expansion
$g(x_{j}) - g(x_{j-1}) = \int_0^1 g'(z_j(t)) e_j \dd t$,
where $z_{j}(t) =  x_{j-1}+te_j$.
Rewriting \eqref{eqn:wk001} with \eqref{eqn:wk002} with this expansion yields
\begin{align}\label{eqn:wk004}
w_{k+1}  
  = (1-\beta_{k-1}) w_k^\alpha +
  \sum_{j = k-m_{k-1}}^{k} \gamma^k_{j} \int_0^1 g'(z_j(t)) e_j \dd t.
\end{align}
Adding and subtracting consecutive averages, each summand of the last term of 
\eqref{eqn:wk004} can be written as
\begin{align}\label{eqn:wk004int1}
\int_0^1 g'(z_j(t))e_j \dd t = 
\int_0^1 g'(z_k(t))e_j \dd t + 
\sum_{n=j}^{k-1} \int_0^1 g'(z_n(t))e_j - g'(z_{n+1}(t)) e_j \dd t.
\end{align}
Summing over the $j$'s,
the sum on the right hand side of \eqref{eqn:wk004} may be expressed as
\begin{align}\label{eqn:wk005}
\sum_{j = k-m_{k-1}}^{k} \!\! \gamma^k_{j} \int_0^1 g'(z_j(t))e_j \dd t
&=\int_0^1 g'(z_k(t)) \sum_{j = k-m_{k-1}}^k  \!\! \gamma^k_j e_j \dd t 
\nonumber \\&
+ \sum_{j = k-m_{k-1}}^{k-1} 
  \sum_{n=j}^{k-1} \gamma^k_j \int_0^1 g'(z_n(t))e_j- g'(z_{n+1}(t)) e_j \dd t. 
\end{align}
From $\sum_{j = k-m_{k-1}}^k \gamma_j^k e_j = x_k - \xa_{k-1}$ 
(see \cite[Section 2]{EPRX19} for details) and \eqref{eqn:update-v1} 
it holds that
\begin{align}\label{eqn:wk007}
\sum_{j=k-m_{k-1}}^k \gamma^k_{j}e_j = 
x_k - \xa_{k-1} = \beta_{k-1} \wa_k.
\end{align}
Putting \eqref{eqn:wk007} together with \eqref{eqn:wk005} and \eqref{eqn:wk004}
then yields
\begin{align}\label{eqn:wk008}
w_{k+1} = \int_0^1 & (1-\beta_{k-1})\wa_k + \beta_{k-1} g'(z_k(t)) \wa_k  \dd t
+ \sum_{j = k-m_{k-1}}^{k-1}  
  \sum_{n=j}^{k-1} \int_0^1 \big(g'(z_n(t)) - g'(z_{n+1}(t)) \big) 
e_j \gamma_j^k \dd t.
\end{align}
Notice that the summands on the right of \eqref{eqn:wk008} are all zero
if $g$ is a linear operator, as $g'$ is then constant.  
The terms summed over are next bounded using
Assumption \ref{assume:g}.  It is worth noting here that for linear operators, 
this will result in zero contribution from higher-order terms, whereas for nonlinear
operators, the higher-order terms are scaled by $\hat \kappa_g>0$, the Lipschitz 
constant of $g'$.  Intuitively, this expansion leads to a local result, as when 
the difference between iterates is sufficiently small, the graph of a function 
satisfying Assumption \ref{assume:g} at or between those iterates is nearly linear.

Taking norms in \eqref{eqn:wk008} and applying Assumption \ref{assume:g} then
triangle inequalities applied to the terms of the final sum produces the expansion of 
$\nr{w_{k+1}}$ in terms of $\nr{\wa_k}$ and $\nr{e_j}$, $j = k-m_k, \ldots, k$, by
\begin{align}\label{eqn:wk009}
\nr{w_{k+1}} &\le \left((1-\beta_{k-1}) + \kappa_g \beta_{k-1}\right) \nr{\wa_k} 
+ \f{\hat \kappa_g}{2} \sum_{j = k-m_{k-1}}^{k-1} \!\! \nr{e_j \gamma^k_j}
  \sum_{n=j}^{k-1} \left( \nr {e_{n+1}} + \nr{e_{n}} \right)
\nonumber \\
&= \left((1-\beta_{k-1}) + \kappa_g \beta_{k-1}\right) \nr{\wa_k} 
+ \f{\hat \kappa_g}{2} 
  \sum_{n=k-m_{k-1}}^{k-1}  \!\!\!\!  \left( \nr {e_{n+1}} + \nr{e_{n}} \right)
\sum_{j = n}^{k-1} \nr{e_j \gamma^k_j},
\end{align}
where the last equality follows from reindexing the sums.
The next step is to bound the $\nr{e_j}$ terms by $\nr{w_j}$ terms.
Here the analysis departs from that in \cite{EPRX19}.
This will be shown first in the simpler case of depth $m=1$ in Section \ref{sec:m1}, and then extended to more
general depths $m > 1$ in Section \ref{sec:genm}.
\section{Acceleration for depth $m=1$}\label{sec:m1}
For depth $m=1$, the matrix $F_k$ has only one column, which removes several 
technicalities from the analysis.  
It is useful to use this case to overview the general
framework and to introduce the extension to a noncontractive setting.
\begin{lemma}\label{lem:m1}
Let Assumption \ref{assume:g} hold, and let $m=1$ in Algorithm \ref{alg:anderson}.  
Assume there is a constant $\sigma > 0$ for which the residuals on stages 
$j+1$ and $j$ satsify Assumption \ref{a:key}.
Then the following bound holds on the difference between consecutive 
accelerated iterates.
\begin{align}\label{eqn:m1-lem}
\nr{e_{j+1}} \le
\nr{w_{j+1}}\left( \sigma^{-1} \sqrt{1-\theta_{j+1}^2} + \beta_j \theta_{j+1} \right).
\end{align}
\end{lemma}

\begin{proof}
The update \eqref{eqn:update-v1} for the case $m=1$ is
\begin{align}\label{eqn:m1-001}
x_{j+1} = (1-\gamma_j^{j+1}) x_j + \gamma_j^{j+1}(x_{j-1}) + \beta_j \wa_{j+1},
\end{align}
where consistent with \eqref{eqn:gammadef}, $\gamma_j^{j+1} = \alpha^{j+1}_{j-1}$.  
Taking norms and applying \eqref{eqn:thetak}
allows
\begin{align}\label{eqn:m1-002}
\nr{e_{j+1}} = \nr{x_{j+1} -x_j} 
 & \le  |\gamma_j^{j+1}| \nr{e_j} + \beta_j \theta_{j+1} \nr{w_{j+1}}.
\end{align}
Inequality \eqref{eqn:m1-002} will be used to trade terms of the form $\nr{e_{j+1}}$
for expressions in terms of $\nr{w_{j+1}}$.
The argument follows by relating the
optimization coefficient $\gamma_j^{j+1}$ to the optimization gain $\theta_{j+1}$.

For $m=1$, the coefficient $\gamma^{j+1}_j$ can be explicitly written as
\begin{align}\label{eqn:m1-002g}
\gamma^{j+1}_j = \f{(w_{j+1},w_{j+1} - w_j)}{\nr{w_{j+1} - w_{j}}^2}.
\end{align}
In particular, this determines the decomposition of $w_{j+1}$ into 
$w_R = \gamma^{j+1}_j (w_{j+1}-w_j)$, in the range of $(w_{j+1}-w_j)$, and 
$w_N = \wa_{j+1} = w_{j+1} - \gamma^{j+1}_j (w_{j+1} - w_j)$, in the nullspace of
$(w_{j+1}-w_j)^\top$.  By the orthogonality of $w_R$ and $w_N$ it follows that
\begin{align}\label{eqn:m1-002RN}
\nr{w_{j+1}}^2 = \nr{w_R}^2 + \nr{w_N}^2 
= \nr{\gamma^{j+1}_j (w_{j+1}- w_j)}^2 + \nr{\wa_{j+1}}^2 
= (\gamma^{j+1}_j)^2 \nr{w_{j+1}-w_j}^2 + \theta_{j+1}^2\nr{w_{j+1}}^2, 
\end{align}
by which
\begin{align}\label{eqn:gamma_theta}
|\gamma_j^{j+1}| 
  = \sqrt{1-\theta_{j+1}^2}\f{\nr{w_{j+1}}}{\nr{w_{j+1}-w_j}},
~\text{ and }~
\theta_{j+1} 
= \sqrt{1 - \f{(w_{j+1}, w_{j+1}-w_j)^2}{\nr{w_{j+1}}^2\nr{w_{j+1 -w_j}}^2}},
\end{align}
where the expression for $\theta_{j+1}$ in \eqref{eqn:gamma_theta} can be recognized
as the (absolute value of the) direction sine between $w_{j+1}$ and $w_{j+1}-w_j$.
Applying the expression for $\gamma^{j+1}_j$ in 
\eqref{eqn:gamma_theta} to \eqref{eqn:m1-002} yields
\begin{align}\label{eqn:m1-003}
\nr{e_{j+1}} \le  \sqrt{1-\theta_{j+1}^2}\f{\nr{w_{j+1}}}{\nr{w_{j+1}-w_j}} 
                   \nr{e_j} + \beta_j \theta_{j+1} \nr{w_{j+1}}.
\end{align}

Applying now they key inequality \eqref{eqn:m1key} to \eqref{eqn:m1-003} yields
\begin{align}\label{eqn:m1-004}
\nr{e_{j+1}} \le  \sigma^{-1} \sqrt{1-\theta_{j+1}^2} \nr{w_{j+1}}
               + \beta_j \theta_{j+1} \nr{w_{j+1}},
\end{align}
establishing the result \eqref{eqn:m1-lem}
\end{proof}

\begin{remark}\label{rem:m1Jposdef}
In the second case of Remark \ref{rem:m1} where $f'$ is nondegenerate, 
the results of Lemma \ref{lem:m1} and $0 < \beta_j \le 1$ show
\[
\nr{e_{j+1}} \le \left( \f{2}{\sigma_f}\sqrt{1-\theta_{j+1}^2} + \theta_{j+1}\right)
\nr{w_{j+1}}
\le \sqrt{1 + 4/\sigma_f^2} \nr{w_{j+1}},
\]
where the last bound was obtained by maximizing the previous expression with respect
to $\theta_{j+1}$. Setting this expression no greater than $\sigma_f/\hat \kappa_g$
it follows that
$\nr{w_{j+1}} \le \sigma_f^2/\big(\hat \kappa_g \sqrt{\sigma_f^2 + 4} \big)$
is sufficient to ensure $\nr{e_{j+1}} \le \sigma_f/\hat \kappa_g$, which implies
satisfaction of Assumption \ref{a:key} on the subsequent iteration.
\end{remark}

Relation \eqref{eqn:m1-004} is now used in the expansion of the 
residual \eqref{eqn:wk009} to bound $\nr{w_{k+1}}$.

\begin{theorem}\label{thm:m1}
Suppose the hypotheses of Lemma \ref{lem:m1} for $j = k-1$ and $j = k-2$. 
Then the following bound holds for the nonlinear residual $\nr{w_{k+1}}$
generated by Algorithm \ref{alg:anderson} with depth $m=1$.
\begin{align}\label{eqn:m1-thm}
&\nr{w_{k+1}}  \le
 \nr{w_k} \bigg\{
\theta_k \big((1-\beta_{k-1}) + \kappa_g \beta_{k-1} \big)
+ \hat \kappa_g \sigma^{-1} \sqrt{1-\theta_{k}^2}
\nonumber \\ & 
  \times\bigg( \nr{w_k}
  \Big(\sigma^{-1} \sqrt{1-\theta_{k}^2} + \beta_{k-1}\theta_{k} \Big)
  +\nr{w_{k-1}}
  \Big(\sigma^{-1} \sqrt{1-\theta_{k-1}^2} + \beta_{k-2}\theta_{k-1} \Big) 
  \bigg)  \bigg\}.
\end{align}
\end{theorem}

\begin{remark}\label{rem:kg1}
Since $\hat \kappa_g$ represents the Lipschitz constant of $g'$, if $g$ is linear then 
$\hat \kappa_g=0$ and thus all of the higher order terms on the right side of 
\eqref{eqn:m1-thm} will vanish.
\end{remark}

This result shows not only how the first order term is scaled by the optimization
gain $\theta_k$, but also that the higher order terms are scaled by 
$\sqrt{1 - \theta_k^2}$. This explicitly establishes that if $\theta_k = 1$, 
then the higher order terms do not contribute to the total residual and the bound
for the fixed-point iteration is recovered. This holds as well for the case $m>1$, 
shown in the next section.

\begin{proof}
Expanding the residual by \eqref{eqn:wk009} yields for depth $m=1$
\begin{align*}
\nr{w_{k+1}} & \le
\theta_k \big((1-\beta_{k-1}) + \kappa_g \beta_{k-1} \big) \nr{w_k} 
+ \hat \kappa_g \left( \nr{e_k} + \nr{e_{k-1}} \right)|\gamma_{k-1}^k|\nr{e_{k-1}},
\end{align*}
where consistent with \eqref{eqn:gammadef}, $\gamma_{k-1}^k = \alpha^k_{k-2}$.

Applying \eqref{eqn:m1-004} with $j= k-1$ and $j=k-2$
allows
\begin{align}\label{eqn:m1-006}
\nr{w_{k+1}} & \le
\theta_k \big((1-\beta_{k-1}) + \kappa_g \beta_{k-1} \big) \nr{w_k} 
+ \hat \kappa_g \bigg( \nr{w_k}
  \Big(\sigma^{-1} \sqrt{1-\theta_{k}^2} + \beta_{k-1}\theta_{k} \Big)
\nonumber \\
  & +\nr{w_{k-1}}
  \left(\sigma^{-1} \sqrt{1-\theta_{k-1}^2} + \beta_{k-2}\theta_{k-1} \right) 
  \bigg)|\gamma_{k-1}^k|\nr{e_{k-1}}.
\end{align}

Combining relation \eqref{eqn:gamma_theta} with hypothesis \eqref{eqn:m1key}
yields
$
|\gamma_{k-1}^k|\nr{e_{k-1}} \le \sigma^{-1}\sqrt{1-\theta_{k}^2} \nr{w_k},
$
by which 
\begin{align}\label{eqn:m1-007}
\nr{w_{k+1}} & \le
\theta_k \big((1-\beta_{k-1}) + \kappa_g \beta_{k-1} \big) \nr{w_k} 
+ \hat \kappa_g \bigg( \nr{w_k}
  \Big(\sigma^{-1} \sqrt{1-\theta_{k}^2} + \beta_{k-1}\theta_{k} \Big)
\nonumber \\
  & +\nr{w_{k-1}}
  \Big(\sigma^{-1} \sqrt{1-\theta_{k-1}^2} + \beta_{k-2}\theta_{k-1} \Big) 
  \bigg)\sigma^{-1} \sqrt{1-\theta_{k}^2} \nr{w_k},
\end{align}
establishing the result \eqref{eqn:m1-thm}.
\end{proof}

The bound \eqref{eqn:m1-thm} shows for $\theta_{k}$ small, the higher-order
terms have a greater contribution whereas for $\theta_{k}$ close to unity
(the optimization did little), the residual is dominated by the first order term; 
and, $\hat \kappa_g$, the Lipschitz constant of $g'$, has less influence on the 
the residual.

In light of Remark \ref{rem:m1}, the two presented conditions
under which the hypothesis \eqref{eqn:m1key} must hold are now discussed.
First, if $g$ is contractive on $X$, then \eqref{eqn:m1key} continues to hold on 
subsequent iterates without further conditions. Moreover in that case it makes
sense to run the iteration without damping ($\beta_j = 1$ for all $j$)
and \eqref{eqn:m1-007} reduces to
\begin{align*}
\nr{w_{k+1}} & \le
 \nr{w_k} \Bigg\{ \theta_k \kappa_g 
+ \f{\hat \kappa_g \sqrt{1-\theta_{k}^2}}{(1-\kappa_g)^2}
\Bigg(\nr{w_k} \Bigg({\sqrt{1-\theta_{k}^2}} + \theta_{k} \Bigg)
+\nr{w_{k-1}}
  \Bigg(\sqrt{1-\theta_{k-1}^2} + \theta_{k-1} \Bigg) 
  \Bigg)  \Bigg\}.
\end{align*}

If instead, 
$\nr{f'(y)(x-y)} \ge \sigma_f\nr{x-y}$ for all $x,y \in X$,
then at the next iteration 
$\nr{w_{k+1}- w_k} \ge (\sigma_f/2)\nr{e_k}$ continues to hold if 
$\nr{e_{k+1}} \le \nr{e_k}$, which is guaranteed upon sufficient decrease of the 
sequence of residuals $\{\nr{w_k}\}$.
This explains the observation (demonstrated by the steady examples of \cite{LWWY12}, 
for instance) that Anderson accelerated noncontractive iterations can show rapid 
convergence.  However, this does not guarantee convergence without some ability
to enforce an inequality such as
$\theta_k((1-\beta_{k-1}) + \kappa_g \beta_{k-1}) < 1-\eps$, with sufficient frequency. 
As $\theta_k$ sufficiently less than one is essential to the success
of the algorithm, this encourages the consideration of the theory for $m >1$
in the next sections, as smaller gain factors can be obtained (to some extent)
with greater algorithmic depth.

Finally, a corollary to \eqref{thm:m1} shows a simplified residual bound for
contractive operators together with a condition for monotonic decrease of the residual.
This result features tighter bounds on the
higher order terms than in \cite{EPRX19}, and without assumptions on the
boundedness of the optimization coefficients.
\begin{corollary}\label{cor:m1con}
Suppose the hypotheses of Lemma \ref{lem:m1} for $j = k-1$ and $j = k-2$, and
the Lipschitz constant of $g$ satisfies $\kappa_g < 1$.
Then the following bound holds on the nonlinear residual $\nr{w_{k+1}}$ generated by
Algorithm \ref{alg:anderson} with $m=1$ and $\beta_k = \beta = 1$:
\begin{align}\label{eqn:m1conb}
&\nr{w_{k+1}}  \le
 \nr{w_k} \bigg\{
\theta_k \kappa_g 
+ \f{ \sqrt 2 \hat \kappa_g}{(1-\kappa_g)^2} \sqrt{1-\theta_{k}^2}
  \bigg( \nr{w_k} +\nr{w_{k-1}}
  \bigg)  \bigg\}.
\end{align}

After the first two consecutive iterations $j = k-1,k$ where
the following inequality is satisfied
\begin{align}\label{eqn:m1con}
\nr{w_{j}}+\nr{w_{j-1}} < \f{\sqrt{1-\kappa_g^2}(1-\kappa_g)^2}{\sqrt 2 \hat \kappa_g},
\end{align}
monotonic decrease of the residual is ensured.
\end{corollary}

\begin{proof}
From \eqref{eqn:m1-thm} with $\beta_k = 1$ and $\sigma = (1-\kappa_g)$, 
the residual $\nr{w_{k+1}}$ satisfies
\begin{align}\label{eqn:m1cor001}
&\nr{w_{k+1}}  \le
 \nr{w_k} \bigg\{
\theta_k \kappa_g 
+ \f{\hat \kappa_g \sqrt{1-\theta_{k}^2}}{(1-\kappa_g)^2} 
\bigg( \nr{w_k}
  \Big(\sqrt{1-\theta_{k}^2} + \theta_{k} \Big)
  +\nr{w_{k-1}}
  \Big(\sqrt{1-\theta_{k-1}^2} + \theta_{k-1} \Big) 
  \bigg)  \bigg\}.
\end{align}
The maximum of $\sqrt{1-\theta^2} + \theta$ on $0 \le \theta \le 1$
is $\sqrt{2}$, attained at $\theta= 1/\sqrt{2}$. Applying this to 
$\theta_{k-1},\theta_k$ within the higher order terms yields \eqref{eqn:m1conb}.

Following the same idea, maximizing the bracketed term on the right hand side of
\eqref{eqn:m1conb} over $\theta_k$ 
\[
\theta_k \kappa_g 
+ \f{ \sqrt 2 \hat \kappa_g}{(1-\kappa_g)^2} \sqrt{1-\theta_{k}^2}
  \bigg( \nr{w_k} +\nr{w_{k-1}}
  \bigg)  \le 
\sqrt{\kappa_g^2 + \frac{2 \hat \kappa_g^2}{(1 - \kappa_g)^4} 
\big(\nr{w_k} + \nr{w_{k-1}}\big)^2}.
\]
Setting (the square of) the right-hand side expression less than one,
it follows that $\nr{w_{k+1}} < \nr{w_k}$ under condition \eqref{eqn:m1con}.
If this condition is satisfied for two consecutive iterates, then 
$\nr{w_{k+1}} < \nr{w_k}$ and $\nr{w_k} < \nr{w_{k-1}}$, which is sufficient 
to ensure monotonic decrease of the sequence.
\end{proof}

This corollary quantifies (in the contractive setting) the transition from the 
preasymptotic regime where the residuals may be large, to the asymptotic regime where
the residuals are small enough that the higher order terms ``don't count,''
and previous convergence results such as those in \cite{PRX18} hold (see also 
\cite{K18,ToKe15} for a different but related approach).
This will be generalized in Corollary \ref{cor:gmcon} for algorithmic
depths $m > 1$ where it will be sufficient for a similar condition to hold for
$m+1$ consecutive iterates.  However the monotonicity result holds only for
contractive operators. 
\section{Acceleration for depth $m>1$}\label{sec:genm}
The analysis for $m > 1$ is somewhat more complicated than for $m=1$, if only
because in the optimization problem for $m=1$, the matrix $F_k$ has only one
column. For $m > 1$, the columns of $F_k$ are in general not orthogonal, and the 
estimates that follow show how detrimental this lack of orthogonality can be to the
convergence rate.  
First some standard results from numerical linear algebra are
recalled. Then, Theorem \ref{thm:m1} is generalized to $m > 1$.

\begin{proposition}Let $R_j$ be a $j\times j$ upper triangular matrix given by
\[
R_j = \left( \begin{array}{cccc} R_{j-1} & b_j \\ 0 & r_{jj} \end{array} \right),
\]
where $R_{j-1}$ is an invertible $j-1 \times j-1$ upper triangular matrix, 
$b_{j}$ is a $j-1 \times 1$ vector of values, and $r_{jj} \ne 0$.  
Then $R_j$ is invertible and the inverse matrix satisfies
\[
R_j^{-1}
= \left( \begin{array}{cccc} R_{j-1}^{-1} & c_j \\ 0 & r_{jj}^{-1} \end{array} \right).
\]

\end{proposition}\label{prop:tridiag-inv}

The next two results are specific to the economy (or thin) 
QR decomposition of $n \times m$ matrix $A$ (see, for instance 
\cite[Chapter 5]{GoVL96}). The following notation will be used throughout 
the remainder of this section.
For $u,v \in \R^n$, let $\cos(u,v) = (u,v)/(\nr{u}\nr{v})$ be the usual direction
cosine between vectors $u$ and $v$, with the corresponding direction sine satisfying
$\sin^2(u,v) = 1 - \cos^2(u,v)$.  
Let $\cA_j$ be the subspace of $\R^n$ given by $\cA_j = \spa\{a_1, \ldots, a_j\}$, with
orthogonal basis $\{q_1, \ldots, q_j\}$; 
let $\sin^2(u,\cA_j) = 1-\sum_{i = 1}^j \cos^2(u,q_i)$, 
denote the square of the direction sine between vector $u$ and $\cA_j$.  

\begin{proposition}\label{prop:diag-entries}
Let $\hat Q \hat R$ be the economy QR decomposition of a matrix $A \in \R^{n\times m}$,
$n \ge m$ where $A$ has columns $a_1, \ldots a_m$, and $\hat Q$ has orthonormal columns 
$q_1, \ldots q_m$.
Then 
\begin{align}\label{eqn:diag-entries}
r_{jj}^2 = \nr{a_j}^2 \sin^2(a_{j}, \cA_{j-1}), ~j = 1, \ldots, m.
\end{align}
\end{proposition}
The proof is standard and follows from writing the columns of $\hat Q$ as
$q_j = v_j /\norm{v_j}$ with $v_j = a_j - \sum_{i = 1}^{j-1}(q_i,a_j)q_i$.
Then $r_{jj}^2 = \norm{v_j}^2 = \norm{a_j}^2 - \sum_{i = 1}^{j-1}(q_i,a_j)^2$ from 
orthogonality.  Factoring out $\norm{a_j}^2$ from each term yields the result.

The next technical lemma gives a bound on the elements of $\hat R^{-1}$;
it is proven here (in the appendix) to make the manuscript more self-contained.
\begin{lemma}\label{lem:invR}
Let $\hat Q \hat R$ be the economy QR decomposition of matrix $A \in \R^{n\times m}$,
$n \ge m$, where $A$ has columns $a_1, \ldots a_m$, $\hat Q$ has orthonormal columns 
$q_1, \ldots q_m$, and $\hat R = (r_{ij})$ is an invertible upper-triangular 
$m \times m$ matrix. Let $\hat R^{-1} = (s_{ij})$. 

Suppose there is a constant $0 < c_s \le 1$ such that 
$|\sin(a_j, \cA_{j-1})| \ge c_s, ~j = 2, \ldots, m$, which implies another constant 
$0 \le c_t < 1$ with
$|\cos(a_j, q_{i})|\le c_t, ~j = 2, \ldots, m$ and $i = 1, \ldots, j-1$.
Then it holds that
\begin{align}\label{eqn:invRd}
s_{11} &= \f {1}{\nr{a_1}},
\quad &&s_{ii} \le \f{1}{\nr{a_i}c_s}, ~i = 2, \ldots, m, 
\\ \label{eqn:invR}
|s_{1j}| &\le \f{c_t(c_t+c_s)^{j-2}}{\nr{a_1}c_s^{j-1}},  
~\text{ and }~
&&|s_{ij}| \le \f{c_t(c_t+c_s)^{j-i-1}}{\nr{a_i}c_s^{j-i+1}},  
~\text{ for }
\end{align}
$i= 2, \ldots, m-1$ and  $j = i+1, \ldots,  m$.
\end{lemma}
The constant $c_s> 0$ ensures the full rank of $A$ and essentially bounds $\hat Q$ 
away from degeneracy, assuring sufficient linear independence of its columns.
While the results are simpler in form if the second constant is taken as $c_t=1$, 
the condition $c_s > 0$ implies $c_t < 1$.  By taking this second constant 
into account, the results reflect that if the columns of $A$ are close to (or actually)
orthogonal, then $c_t$ and the off-diagonal elements are close to (or actually)  zero. 

The next lemma generalizes Lemma \ref{lem:m1} to $m > 1$.  The technical difficulty
of the more complicated relationship between the optimization coefficients and
optimization gain is handled by expressing both in terms of a QR decomposition and
then making use of Lemma \ref{lem:invR}.

\begin{lemma}\label{lem:genm}
Let Assumption \ref{assume:g} hold. 
Let $v_{n+1} = w_{n+1}- w_n$, 
and let Assumption \ref{a:key} hold with constant $\sigma$ for
$n = j-m, \ldots, j$.
Further, assume there is a constant $c_s>0$ such that
\begin{align}\label{eqn:gmcond}
|\sin\left(v_{n+1}, \spa\{v_{n+2}, \ldots v_{j+1} \}\right)| &\ge c_s, 
~n = j-m+1, \ldots, j-1,
\end{align}
which implies there is a constant $0 \le c_t < 1$ which satisfies
\[
|\cos\left(v_{n+1}, v_p \}\right)| \le c_t, 
~n = j-m+1, \ldots, j-1, ~\text{ and }~ p = n+2, \ldots, j+1.
\]

Then the following bound holds for the difference between consecutive iterates
${e_{j+1}} = {x_{j+1} - x_j}$:
\begin{align}\label{eqn:lgenm}
\nr{e_{j+1}} \le \nr{w_{j+1}}\left(C_{F,j+1}
\sqrt{1 - \theta_{j+1}^2} + \beta_j \theta_{j+1}\right),
\end{align}
where the constant $C_{F,j+1}$ is given by
\begin{align}\label{eqn:cfdef}
C_{F,j+1} \coloneqq 
\sigma^{-1} \left(
1 + \f{(1+c_t)\sum_{l = 1}^{m_j-1}\binom{m_j-1}{l}c_t^{l-1}c_s^{m_j-l-1}}{c_s^{m_j-1}}
\right).
\end{align}
Additionally, the following bounds hold for terms of the form $\nr{e_n\gamma^{j+1}_n}$.
\begin{align}\label{eqn:lgenm2}
\nr{e_{n}\gamma_{n}^{j+1}} &\le C_{n,j+1}\beta_j\sqrt{1-\theta_{j+1}^2}
\nr{w_{j+1}}, ~ n = j-m_j, \ldots, j,
\end{align}
where the constants $C_{n,j+1}$ are given by
\begin{align}\label{eqn:cndef}
C_{n,j+1} \coloneqq \left\{ \begin{array}{ll}
\f{1}{\sigma}\left(\f{c_t+c_s}{c_s} \right)^{{m_j}-1}, & n = j \\
\f{1}{\sigma c_s}\left(\f{c_t+c_s}{c_s} \right)^{m_j-(j-n+1)}, & n = j-m_j, \ldots, j-1
\end{array}\right. .
\end{align}

\end{lemma}
The additional assumption of \eqref{eqn:gmcond} 
not found in the $m=1$ case requires
that the columns of the matrix used in the least squares problem \eqref{eqn:opt-v1}, 
$v_{j+1}, \ldots, v_{j-m+2}$, maintain sufficient linear independence. 
See Subsection \ref{sec:practical} on ensuring this assumption holds during
a simulation.

\begin{proof}
Throughout this proof, depth $m_j$ will be denoted by $m$, for simplicity.
Starting with the update for $x_{j+1}$ from \eqref{eqn:update-v1} and 
\eqref{eqn:gammadef}, defined for
optimization coefficients $\gamma^{j+1}$ from \eqref{eqn:opt-v1},
and the matrix $E_j$ given by \eqref{eqn:EkFkdef}, shows
$x_{j+1} - x_j = - E_j \gamma^{j+1} + \beta_k\wa_{j+1}. $
Taking norms and applying \eqref{eqn:thetak} yields
\begin{align}\label{eqn:lgm001}
\nr{e_{j+1}} & \le \nr{E_j \gamma^{j+1}} + \beta_j \theta_{j+1} \nr{w_{j+1}}.
\end{align}

By \eqref{eqn:opt-v1}, the coefficients $\gamma^{j+1}$ are the least-squares solution
to $F_j \gamma^{j+1} = w_{j+1}$, where $F_j$ is given by \eqref{eqn:EkFkdef}.  
Using an economy QR-decomposition provides
$\hat R \gamma^{j+1} = \hat Q^\top w_{j+1}$, by which 
\eqref{eqn:lgm001} may be written 
\begin{align}\label{eqn:lgm002}
\nr{e_{j+1}}  
\le  \nr{E_j \hat R^{-1} \hat Q^\top w_{j+1}} + \beta_j \theta_{j+1} \nr{w_{j+1}}. 
\end{align}

The first term on the right of \eqref{eqn:lgm002} can be bounded in terms
of $\nr{w_{j+1}}$ by considering an explicit expression for the optimization gain 
$\theta_{j+1}$, as first discussed in \cite{EPRX19}.
From \eqref{eqn:thetak} and the unique decomposition $w_{j+1} = w_R + w_N$ with
$w_R \in \ran(F_j)$ and $w_N \in \nulll((F_j)^\top)$,
the null-space component $w_N$ is the least-squares residual satisfying
$\nr{w_N} = \nr{F_j \gamma^{j+1} - w_{j+1}} = \nr{\wa_{j+1}} 
= \theta_{j+1} \nr{w_{j+1}}$,
meaning $\theta_{j+1} = \sqrt{1 - \nr{\hat Q^\top w_{j+1}}^2/\nr{w_{j+1}}^2}$, or, 
by rearranging
\begin{align}\label{eqn:lgm003}
\nr{w_{j+1}}\sqrt{1 - \theta_{j+1}^2}
=\nr{\hat Q^\top w_{j+1}}.
\end{align}

The first term on the right-hand side of \eqref{eqn:lgm002} can now be controlled
by \eqref{eqn:lgm003}, yielding
\begin{align}\label{eqn:lgm004}
\nr{E_j \hat R^{-1} \hat Q^\top w_{j+1}}
\le \nr{E_j \hat R^{-1}}\nr{\hat Q^\top w_{j+1}}
\le \nr{E_j \hat R^{-1}}
\nr{w_{j+1}}\sqrt{1 - \theta_{j+1}^2}.
\end{align}
It remains to bound $\nr{E_j \hat R^{-1}}$. Writing $\hat R^{-1} = S = (s_{ij})$,
\begin{align}\label{eqn:lgm005}
\nr{E_j \hat R^{-1}} &= \nr{\left(\begin{array}{ccccccc} 
e_j \sum_{n =1}^m s_{1n} & e_{j-1} \sum_{n = 2}^m s_{2n} & \cdots &
e_{j-m+1} s_{mm}
\end{array}\right)}
\nonumber \\ & 
\le 
\nr{e_j \sum_{n =1}^m s_{1n}} + \nr{e_{j-1} \sum_{n = 2}^m s_{2n}} + \cdots +
\nr{e_{j-m+1} s_{mm}},
\end{align}
where the last inequality follows from the standard bound of the matrix 2-norm 
by the Frobenius norm. Apply now the results of the technical Lemma \ref{lem:invR}.

For the first term in the sum of vector norms in \eqref{eqn:lgm005},
applying \eqref{eqn:invRd}-\eqref{eqn:invR} of Lemma \ref{lem:invR} then taking the
finite geometric sum produces the bound
\begin{align}\label{eqn:lgm006}
\nr{e_j \sum_{n =1}^m s_{1n}} & \le
\nr{e_j}\left| \sum_{n = 1}^m s_{1n} \right|
\nonumber \\
& \le \f{\nr{e_j}}{\nr{w_{j+1}-w_j}}\left( 1 + 
\sum_{n = 2}^m \f{c_t(c_t + c_s)^{n-2}} {c_s^{n-1}} \right)
\nonumber \\
& = \f{\nr{e_j}}{\nr{w_{j+1}-w_j}} 
\left(\f{c_t + c_s}{c_s} \right)^{m-1}
\nonumber \\
& \le \sigma^{-1}\left( \f{c_t+c_s}{c_s}\right)^{m-1},
\end{align}
where the last inequality follows from the hypothesis \eqref{eqn:m1key}.

Proceed similarly for the remaining vector norms of \eqref{eqn:lgm005}, indexed by
$p = 2, \ldots, m$, noting the additional factor of $c_s$ in the denominator, to get
\begin{align}\label{eqn:lgm007}
\nr{e_{j-p+1} \sum_{n = p}^m s_{pn}} & \le
\f{1}{\sigma c_s}\left(1 + \sum_{n = p+1}^m \f{(c_t+c_s)^{n-(p+1)}}{c_s^{n-p}}  \right)
\le
\f{1}{\sigma c_s}\left( \f{c_t+c_s}{c_s}\right)^{m-p}.
\end{align}
Finally, adding the contributions from $p = 1$ to $p = 2, \ldots, m$ from 
\eqref{eqn:lgm006} and \eqref{eqn:lgm007} and applying the total to \eqref{eqn:lgm005}
yields, assuming $c_t \ne 0$
\begin{align}\label{eqn:lgm008}
\nr{E_j \hat R^{-1}} & \le \sigma^{-1} \left(\f{(c_t+c_s)^{m-1}(c_t + 1) - c_s^{m-1}}
{c_s^{m-1}c_t}\right)
= \sigma^{-1} \left(
1 + \f{(1+c_t)\sum_{j = 1}^{m-1}{}\binom{m-1}{j}c_t^{j-1}c_s^{m-j-1}}{c_s^{m-1}}
\right).
\end{align}
If it so happens that $c_t= 0$, meaning the columns of $F_k$ are orthogonal, then
$c_s = 1$ and \eqref{eqn:lgm008} is in agreement with summing the terms directly
from \eqref{eqn:lgm006} and \eqref{eqn:lgm007} yields 
$\nr{E_j \hat R^{-1}} \le  m/\sigma$, in agreement in \eqref{eqn:lgm008}. 
Putting \eqref{eqn:lgm008} together with \eqref{eqn:lgm003} yields
\begin{align*}
\nr{e_{j+1}}& \le 
C_{F,j+1} \sqrt{1- \theta_{j+1}^2}\nr{w_{j+1}}
+ \beta_j \theta_{j+1} \nr{w_{j+1}},
\end{align*}
with $C_{F,j+1}$ given by \eqref{eqn:cfdef}, hence the result \eqref{eqn:lgenm}.

For the second result, \eqref{eqn:lgenm2}, expanding \eqref{eqn:lgm002}, shows
\begin{align}\label{eqn:lgm009}
\left(\begin{array}{cccccc}e_j\gamma_j^{j+1} & e_{j-1} \gamma_{j-1}^{j+1}
  & \cdots & e_{j-m+1}\gamma_{j-m+1}^{j+1} \end{array} \right)
 = E_j \gamma^{j+1}
 = E_j \hat R^{-1} \hat Q^\top w_{j+1}.
\end{align}
Accordingly, 
$e_{j-p+1}\gamma^{j+1}_{j-p+1} = e_{j-p+1} s^{p} \hat Q^\top w_{j+1}$,
where $s^p$ is row $p$ of $\hat R^{-1}$. Hence following
\eqref{eqn:lgm006} and applying \eqref{eqn:lgm003} produces for the first column of 
\eqref{eqn:lgm009}:
\begin{align*}
\nr{e_j \gamma^{j+1}_j}  \le \nr{e_j \sum_{n = 1}^{m} s_{1n}} \nr{\wa_{j+1}}
 \le \sigma^{-1} \left(\f{c_t + c_s}{c_s} \right)^{{m}-1} 
\beta_j\sqrt{1-\theta_{j+1}^2}\nr{w_{j+1}}.
\end{align*}
For the remaining columns, following now \eqref{eqn:lgm007} allows
\begin{align*}
\nr{e_{j-p+1} \gamma^{j+1}_{j-p+1}}  
&\le \nr{e_{j-p+1} \sum_{n = 1}^m s_{pn}} \nr{\wa_{j+1}}
\le \f{1}{\sigma c_s} 
\left(\f{c_t + c_s}{c_s} \right)^{{m_j}-p} \beta_j\sqrt{1-\theta_{j+1}^2}\nr{w_{j+1}},
\end{align*}
which establishes the second result \eqref{eqn:lgenm2} with $n = j-p+1$.

\end{proof}	

Lemma \eqref{lem:genm} is now used to establish one-step residual
bounds for general depths $m$.

\begin{theorem}\label{thm:genm}
Suppose the hypotheses of Lemma \ref{lem:genm} for $j = k-m, \ldots, k-1$. 
Then the following bound holds for the nonlinear residual $\nr{w_{k+1}}$
generated by Algorithm \ref{alg:anderson} with depth $m$:
\begin{align}\label{eqn:tgenm}
\nr{w_{k+1}} & \le \nr{w_k} \Bigg\{
 \theta_k ((1-\beta_{k-1}) + \kappa_g \beta_{k-1})
+ \f{\hat \kappa_g}{2}\bigg(
 \nr{w_{k}}h(\theta_{k})h_{k-1}(\theta_k)
\nonumber \\ &
 + 2  \sum_{n = k-{m_{k-1}}+1}^{k-1} 
\Big( \nr{w_n}h(\theta_n) \sum_{j = n}^{k-1} h_j(\theta_k) \Big)
 + \nr{w_{k-m_{k-1}}}h(\theta_{k-m_{k-1}})
  \sum_{j = {k-m_{k-1}}}^{k-1} h_j(\theta_k)
 \bigg) \Bigg\},
\end{align}
where
\begin{align}
h(\theta_j)  = C_{F,j} \sqrt{1 - \theta_j^2} + \beta_{j-1}\theta_j, 
\qquad
h_j(\theta_k)  = C_{j,k} \beta_{k-1}\sqrt{1-\theta_k^2},
\end{align}
and the constants $C_{F,j}$ and $C_{j,k}$ are given by \eqref{eqn:cfdef} and 
\eqref{eqn:cndef}, respectively.
\end{theorem}

\begin{remark}
As in Remark \ref{rem:kg1},
if $g$ is linear then $\hat \kappa_g=0$ and the higher-order terms do not appear.

\end{remark}
\begin{remark}\label{rem:tgenm}
Theorem \ref{thm:genm} gives three significant improvements for the higher order terms, 
compared to the results for general $m$ in \cite{EPRX19}.  First, the results above show
\[
\nr{w_{k+1}} \le \bigo(\nr{w_k}^2) + \bigo(\nr{w_k}\nr{w_{k-1}}) 
  + \ldots \bigo(\nr{w_k}\nr{w_{k-m_{k-1}}}),
\]
whereas previous results show
$\nr{w_{k+1}} \le \bigo(\nr{w_k}^2) + \bigo(\nr{w_{k-1}}^2) 
  + \ldots \bigo(\nr{w_{k-m_{k-1}}}^2).$
This helps to explain how the steady Navier-Stokes numerical test of Section \ref{sec:numerics}
is able to converge with very large $m$.

Second, the theorem makes no assumptions on the boundedness of the 
optimization coefficients.  Instead, a more practical assumption is made for 
how close the matrix $F_k$ from the least-squares problem \eqref{eqn:opt-v1}
comes to degeneracy.
Third, similar to the $m=1$ case of Theorem \ref{thm:m1}, Theorem \ref{thm:genm}
shows the higher order terms do not contribute to the residual if there is no gain
from the optimization problem ($\theta_k = 1$). To see this, note 
that each  $h_j(\theta_k)$ in \eqref{eqn:tgenm} has $\sqrt{1-\theta_k^2}$ as a factor, 
so if there is no gain from the optimization problem, 
then all the higher order
terms in \eqref{eqn:tgenm} vanish.

More explicitly,
each $h_j(\theta_k)$ in \eqref{eqn:tgenm} is bounded by 
$C \sqrt{1-\theta_k^2}$ for a constant $C$ (given
in \eqref{eqn:Cdef}, where the factor of $(1-\kappa_g)$ in the denominator can be 
replaced by $\sigma$ for the general case).
Applying these simplifications to  \eqref{eqn:tgenm} shows $\nr{w_{k+1}}$ satisfies
the bound
\begin{align*}
\nr{w_{k+1}} & \le \nr{w_k} \Bigg(
 \theta_k ((1-\beta_{k-1}) + \kappa_g \beta_{k-1})
+ \f{C\hat \kappa_g\sqrt{1-\theta_k^2}}{2}\bigg(
 \nr{w_{k}}h(\theta_{k})
\nonumber \\ &
 + 2  \sum_{n = k-{m_{k-1}}+1}^{k-1} (k-n)\nr{w_n}h(\theta_n) 
 + m_{k-1}\nr{w_{k-m_{k-1}}}h(\theta_{k-m_{k-1}})
 \bigg) \Bigg).
\end{align*}
\end{remark}

The proof of Theorem \ref{thm:genm} follows the same essential outline as 
Theorem \ref{thm:m1}.
In contrast to the technique used in \cite{EPRX19}, 
a direct rather than inductive approach will be taken, 
as the optimization gain (which depends on $m$) appears in both higher and 
lower order terms.

\begin{proof}
The expansion of the residual \eqref{eqn:wk009} from 
Section \ref{sec:resi-exp} shows
\begin{align}\label{eqn:tgm001}
\nr{w_{k+1}}
\le \left((1-\beta_{k-1}) + \kappa_g \beta_{k-1}\right) \nr{\wa_k} 
+ \f{\hat \kappa_g}{2} 
  \sum_{n=k-m_{k-1}}^{k-1} \!\!\!\! \left( \nr {e_{n+1}} + \nr{e_{n}} \right)
\sum_{j = n}^{k-1} \nr{e_j \gamma^k_j}.
\end{align}

Opening up the first sum of \eqref{eqn:tgm001} allows
\begin{align}\label{eqn:tgm002}
&\sum_{n=k-m_{k-1}}^{k-1} \left( \nr {e_{n+1}} + \nr{e_{n}} \right)
\sum_{j = n}^{k-1} \nr{e_j \gamma^k_j}
\nonumber \\& 
= \nr{e_{k-m_{k-1}}}\sum_{j = {k-m_{k-1}}}^{k-1} \nr{e_j \gamma_j^k}
 + 2 \!\!\!\! \sum_{n = k-{m_{k-1}}+1}^{k-1} \!\!\!\!  \nr{e_n} 
      \sum_{j = n}^{k-1} \nr{e_j \gamma^k_j}
  + \nr{e_{k}}\nr{e_{k-1}\gamma_{k-1}^k}.
\end{align}
Applying now \eqref{eqn:lgenm} then \eqref{eqn:lgenm2}, the above \eqref{eqn:tgm002} 
is bounded by
\begin{align}\label{eqn:tgm003}
 & \nr{w_{k}}h(\theta_{k})\nr{e_{k-1}\gamma_{k-1}^k}
 +\! 2 \!\! \sum_{n = k-{m_{k-1}}+1}^{k-1} \!\! \Big( \nr{w_n}h(\theta_n) 
  \sum_{j = n}^{k-1} \nr{e_j \gamma^k_j}\Big)
+ \nr{w_{k-m_{k-1}}}h(\theta_{k-m_{k-1}}) \!\!
  \sum_{j = {k-m_{k-1}}}^{k-1} \nr{e_j \gamma_j^k}
\nonumber \\
 &\le \nr{w_{k}}h(\theta_{k})h_{k-1}(\theta_k)\nr{w_{k}}
 + 2  \sum_{n = k-{m_{k-1}}+1}^{k-1} \Big( \nr{w_n}h(\theta_n) 
  \sum_{j = n}^{k-1} h_j(\theta_k)\nr{w_k} \Big)
\nonumber \\ & 
+ \nr{w_{k-m_{k-1}}}h(\theta_{k-m_{k-1}})
  \sum_{j = {k-m_{k-1}}}^{k-1} h_j(\theta_k)\nr{w_k}.
\end{align}

Putting the bound of \eqref{eqn:tgm003} back into \eqref{eqn:tgm002} then yields
\begin{align*}
\nr{w_{k+1}} & \le \nr{w_k} \Bigg(
 \theta_k ((1-\beta_{k-1}) + \kappa_g \beta_{k-1})
\nonumber \\ &
+ \f{\hat \kappa_g}{2}\bigg(
 \nr{w_{k}}h(\theta_{k})h_{k-1}(\theta_k)
 + 2  \sum_{n = k-{m_{k-1}}+1}^{k-1} \Big( \nr{w_n}h(\theta_n) 
  \sum_{j = n}^{k-1} h_j(\theta_k) \Big)
\nonumber \\ & 
+ \nr{w_{k-m_{k-1}}}h(\theta_{k-m_{k-1}})
  \sum_{j = {k-m_{k-1}}}^{k-1} h_j(\theta_k)
 \bigg) \Bigg),
\end{align*}
hence the result.
\end{proof}

The next corollary gives conditions to assure the monotonic decrease of the residual, 
in the contractive setting.

\begin{corollary}\label{cor:gmcon}
Suppose the hypotheses of Lemma \ref{lem:genm} for $j = k-m, \ldots, k-1$,
and the Lipschitz constant $\kappa_g$ satisfies $\kappa_g < 1$.
Then the following bound holds for the nonlinear residual
$\nr{w_{k+1}}$ generated by 
Algorithm \ref{alg:anderson} with $\beta_k = \beta =1$.
\begin{align}\label{eqn:gmconb}
\nr{w_{k+1}} & \le \nr{w_k} \Bigg\{
 \theta_k \kappa_g 
+ \f{C\sqrt{1+C^2} \sqrt{1-\theta_k^2}\hat \kappa_g}{2}
\nonumber \\ &
\times\bigg( \nr{w_{k}}
 + 2  \sum_{n = k-{m_{k-1}}+1}^{k-1} (k-n)\nr{w_n}
 + m_{k-1}\nr{w_{k-m_{k-1}}} 
 \bigg) \Bigg\},
\end{align}
where
\begin{align}\label{eqn:Cdef}
C = \max\left\{\frac{1}{\sigma c_s} \left(\frac{c_t + c_s}{c_s} \right)^{m_k-1},
C_{F,k+1}\right\}, ~\text{ with } \sigma = (1-\kappa_g).
\end{align}

After the first $m+1$ consecutive iterations $j = k-m, \ldots, k$, 
(assuming here for simplicity that 
$k \ge 2m$, so the subscript on $m$ may be dropped) such that
the following inequality is satisfied
\begin{align}\label{eqn:gmcon}
\nr{w_j} + 2\sum_{n = j-m+1}^{k-1} \nr{w_n} + \nr{w_{j-m}}
< \frac{2(1-\kappa_g)}{C \sqrt{1 + C^2}\hat \kappa_g},
\end{align}
monotonic decrease of the residual is assured.
\end{corollary}
The proof follows similarly to the $m=1$ case in Corollary \ref{cor:m1con}, 
with the additional steps of bounding the two types of $h$ coefficients.

\begin{proof}
For each $\beta_j = 1$ and $\sigma= 1-\kappa_g$, 
as in Remark \ref{rem:tgenm} the coefficients $h_n(\theta_k)$ are each bounded by 
$C \sqrt{1-\theta_k^2}$, with $C$ given by \eqref{eqn:Cdef}.  
Applying this to \eqref{eqn:tgenm} allows
\begin{align}\label{eqn:gmcor001}
\nr{w_{k+1}} & \le \nr{w_k} \Bigg(
 \theta_k\kappa_g
+ \f{C \hat \kappa_g \sqrt{1 - \theta_k^2}}{2}\bigg(
 \nr{w_{k}}h(\theta_{k})
\nonumber \\ &
 + 2  \sum_{n = k-{m_{k-1}}+1}^{k-1} 
(k-n)\nr{w_n}h(\theta_n)
 + m_{k-1}\nr{w_{k-m_{k-1}}}h(\theta_{k-m_{k-1}})
 \bigg) \Bigg).
\end{align}
The coefficients $h(\theta_j)$ are each bounded by 
$C \sqrt{1-\theta_j^2} + \theta_j \le \sqrt{1+C^2}$. Applying this to 
\eqref{eqn:gmcor001} yields \eqref{eqn:gmconb}. 

Maximizing in terms of $\theta_k$, the square of the 
bracketed terms on the right hand side
of \eqref{eqn:gmconb} is bounded by
\begin{align}\label{eqn:gmcor002}
\kappa_g^2 + \f{C^2(1+C^2) \hat \kappa_g^2}{4}
\left(\nr{w_k} + 2 \sum_{n = k-m+1}^{k-1}(k-n)\nr{w_n} + m \nr{w_{k-m}}\right)^2.
\end{align}
Setting \eqref{eqn:gmcor002} less than one implies $\nr{w_{k+1}} < \nr{w_k}$ 
under the condition \eqref{eqn:gmcon}. Satisfaction of $\nr{w_{j+1}}< \nr{w_j}$ 
for $m+1$ consecutive iterates $j = k-m, \ldots, k$, then implies reduction in 
every subsequent residual.
\end{proof}

\subsection{Practical guidance based on the theory}\label{sec:practical}
The results of Theorems \ref{thm:m1} and \ref{thm:genm} and Corollaries 
\ref{cor:m1con} and \ref{cor:gmcon} indicate that the most effective choice of algorithmic
depth $m = m_k$
may be to have it increase through the simulation based on the three following regimes.
The different regimes below can depend on the scaling of
data and choice of initial iterates.
Damping (not explicitly discussed here; see for instance
\cite{EPRX19}) may be necessary to see a reduction in the first order residual terms
in noncontractive settings, particularly in the initial regime.
It is assumed here that the problem dimension $n$ is significantly larger than the number
of iterations allowed to solve the problem, and a ``large'' value of $m_k$ is still small 
compared with $n$. The following 3-phase and 2-phase approaches are demonstrated in
the numerical experiments that follow.

\subsubsection{3-phase approach}\label{subsubsec:3phase}
This approach is appropriate for problems where the
initial residual is large or poorly scaled, 
such that an accumulation of higher-order terms can cause lack of convergence, 
or even an overflow. 
This method is demonstrated on the $p$-Laplacian in Section \ref{sec:numerics}.
\begin{itemize}
\item Initial regime: The residual $w_k$ and difference between iterates $e_k$ may be
large (in norm), for instance $\bigo(1)$ or greater.  
The depth $m_k$ should be chosen small (for instance between 0 and 2),
as the accumulation of higher order terms on the right-hand side of \eqref{eqn:tgenm}
({\em cf.} \eqref{eqn:wk008})
can prevent convergence. Additionally, as shown in \eqref{eqn:m1-thm} of 
Theorem \ref{thm:m1} and \eqref{eqn:tgenm} of Theorem \ref{thm:genm}, 
a more successful optimization gives greater weight to the higher-order terms.

\item Pre-asympototic regime. The residual or difference between iterates is on the order
of $10^{-1}$ or $10^{-2}$. The depth $m_k$ can safely be increased roughly 
logarithmically with  $\nr{w_k}$, either to convergence tolerance or 
until the asymptotic regime is reached. 

\item Asympototic regime. The residual is sufficiently small so that the higher-order 
terms of \eqref{eqn:tgenm} are negligible, regardless of their scaling with respect to
the optimization gain $\theta_k$.
The depth $m_k$ can be increased, but should be only to the point that it has
an impact on decreasing the gain $\theta_k$.  
\end{itemize}

Notably, in the results shown below for the steady Navier-Stokes equations,
simply choosing a large depth $m$ (by which $m_k = k-1$, either up to some maximum
$m$, or up to convergence) can be effective in the case that the initial residual
is not greater than $\bigo(1)$, and drops sufficiently rapidly through the initial
iterations.  This is essentially the strategy above, starting in the pre-asymptotic,
rather than the initial regime.  In this example it is also demonstrated that switching
to Newton iterations upon sufficient decay of the residual can yield rapid convergence.

\subsubsection{2-phase approach} This method is appropriate for problems with a 
moderately scaled initial residual, on the order of $\bigo(1)$, and
is demonstrated in Section \ref{sec:numerics} on a nonlinear Helmholtz equation.
\begin{itemize}
\item Pre-asymptotic regime. The depth $m_k$ is kept at a small to moderate value
(2 to 5), until the residual drops below a given threshold, on the order of 
$10^{-2}$ or $10^{-3}$.
\item Asymptotic regime. The depth $m_k$ is increased to a higher steady level,
for instance $m_k = 10$.  This allows smaller factors of the optimization gain
$\theta_k$ due to a better solution of the least-squares problem at the point
where the residual is small enough that the increased weight and accumulation of
higher-order terms does not interfere with convergence.
\end{itemize}

\subsubsection{Safeguarding and verification of \eqref{eqn:gmcond} 
on sufficient linear independence of
the columns of $F_k$}\label{subsubsec:enforce}
It makes sense both numerically and theoretically to solve the least squares problem
at each stage by means of an economy QR decomposition.  
Given the large number of degrees of freedom $n$ in comparison to the algorithmic
depth $m$, forming the decomposition and solving the least squares system  has 
a negligible effect on total solution time; in each of our examples given below, the
total runtime is dominated by the linear system solves.
Let $F_k = \hat Q \hat R$, where $\hat R$ is $m_k \times m_k$. 
Denoting the columns of $F_k$ by $\{v_1, \ldots, v_{m_k}\}$,
Proposition \ref{prop:diag-entries} shows
the diagonal values of $\hat R$ are given by 
$r_{ii} = \nr{v_i}|\sin(v_i,\spa\{v_1, \ldots, v_{i-1}\}|)$, 
by which 
$|\sin(v_i,\spa\{v_1, \ldots, v_{i-1}\})| = r_{ii}/\nr{v_i}$.
If a practitioner wishes to enforce \eqref{eqn:gmcond}, any column $i$ for which this
quantity falls below a given threshold may be removed from $F_k$, (and accordingly
from $E_k$) and the QR decomposition recalculated (or dynamically updated), 
by which \eqref{eqn:gmcond} is satisfied in accordance with that threshold
({\em cf.} the safeguarding strategy introduced in \cite[Section 2.1.2]{PoSc20} 
for AA with $m=1$ applied to Newton iterations).  
The method is well defined for any threshold value in $c_T \in [0,1)$ providing no 
safeguarding for $c_T = 0$ and otherwise enforcing $c_s = c_T$.
This method ensures the most recent column of $F_k$ is used since
$r_{11}/\nr{v_i} = 1$.
This strategy is demonstrated below in Section \ref{sec:numerics} 
on a finite element discretization of the $p$-Laplace equation, with $p$ close to one.

This safeguarding strategy may be compared with that used in \cite[Section 4]{WaNi11} and 
\cite{YMLW09}, in which the
condition of $F_k$ is monitored by the condition of $\hat R$ (which is the same), and the
oldest column of $F_k$ is dropped if the condition exceeds a subscribed threshold.
The main difference is the present method allows the efficient numerical determination of
which column(s) to drop, yielding a theoretically sound update to this older 
heuristic method.
An alternate strategy based on monitoring the condition number of $\hat R$, as suggested
in \cite{FaSa09}, is to 
compute the singular value decomposition (SVD) of $\hat R$, and compute the 
least-squares solution using the pseudoinverse of the truncated expansion to preserve 
the condition (see \cite{chan82,sidi16}).  
However, as shown in the numerical examples of \cite[Section 3.1-3.2]{ToKe15}, 
the deteriorating condition of the least-squares matrix does not necessarily 
interfere with convergence. 

\section{Numerical Experiments}\label{sec:numerics}
In this section, the following test problems will illustrate the theory 
and the practical guidance given above demonstrating both safeguarding
and dynamic depth selection stragies, 
extend the AA methodology to a new application in the nonlinear Helmholtz equation, 
and improve on existing results for AA applied to the steady Navier-Stokes equations.

\subsection{$p$-Laplace equation} 

The $p$-Laplace (or $p$-Poisson) equation arises in many physical 
applications, including non-Newtonian flows {\em e.g.,} in glaciology; turbulent
flows, and flows in porous media; see \cite{DiTh94,GlRa03,DFTW20}.
The elliptic $p$-Laplace equation which is given by 
\[
-\div \left( (|\grad u|^2/2)^{(p-2)/2} \grad u\right) = f,
\]
is degenerate  
for $p > 2$ and singular for $1 < p < 2$.  
In \cite{EPRX19}, the $p$-Laplace equation with $p > 2$ is used to demonstrate an 
approach to adaptively updating damping factors $\beta_k$, and it is used as a 
benchmark problem to demonstrate preconditioned nonlinear solvers in \cite{BKST15}.
For this example, consider a regularized version in the singular regime
\begin{align}\label{eqn:pLap-strong}
-\div\left( \left(\eps^2 + \f 1 2 |\grad u|^2\right)^{(p-2)/2} \grad u \right) = c,
\end{align}
with $\eps = 10^{-14}$, $p = 1.06$ ({\em cf.} \cite{DFTW20}), and $c = \pi$,
over domain $(0,2) \times (0,2)$,
subject to homogeneous Dirichlet boundary conditions 

The results shown below use a $P_1$ finite element discretization over a
$256 \times 256$ uniform triangulation of the domain, which
produces a discrete nonlinear problem with 66,049 degrees of freedom.  
The simulations were run using a
Python implementation of the FEniCS finite element library, \cite{fenics2015}.
Each simulation was started from initial iterate
$u_0 = xy(x-1)(y-1)(x-2)(y-2)$, ({\em cf.,} \cite{BKST15})
and run to a residual tolerance of $\nr{w_k} < 10^{-10}$, where and $l_2$ norm is 
used for both the convergence tolerance and the optimization.

A Picard (fixed-point) iteration for the $P_1$ finite element
discretization of the variational form of \eqref{eqn:pLap-strong} is given by:
Find $u_k\in V_h$ satisfying for all $v\in V_h$
\begin{align} \label{eqn:pLap1}
\int_\Omega \left(\eps^2 + \f 1 2 |\grad u_{k-1}|^2\right)^{(p-2)/2}\grad u_k\cdot
\grad v \dd x = \int_\Omega cv \dd x,
\end{align}
where $V_h$ is the space of continuous piecewise linear functions
that vanish on the boundary.

With the given parameters, the defined fixed-point operator is not globally contractive, 
but the Picard iteration does converge essentially linearly as it approaches the solution.
Here, AA is applied with $m_k = \min\{k-1,m\}$ for $m = 0,2,4,8$, and these results are
compared with dynamically updating the depth $m_k$ by defining
\begin{align}\label{eqn:3phase}
\tilde m_k \coloneqq \ceil(-\log_{10} \nr{w_{k+1}}), \quad 
\text{ and }~  m_k = \psi_{n,N}(\tilde m_k) \coloneqq 
\left\{\begin{array}{cc}
n, & \tilde m_k \le n \\ 
\tilde m_k, & n < \tilde m_k < N \\
N, & \tilde m_k \ge N
\end{array} \right. .
\end{align} 
Setting the depth by \eqref{eqn:3phase} implements the 3-phase approach described in
Subsection \ref{subsubsec:3phase}.
Additionally, results are shown for enforcing the sufficient linear-independence
condition \eqref{eqn:gmcond} holds, with threshold
$c_s = 0.25$, by the method described in Subsection \ref{subsubsec:enforce}. 
Runs where this condition is enforced are denoted as safeguarded (SG) in Figures
\ref{fig:pLap246} and \ref{fig:pLap018}, below. 

\begin{figure}[H]
\begin{center}
\includegraphics[trim = 0pt 0pt 10pt 10pt, clip=true, width = .49\textwidth]
{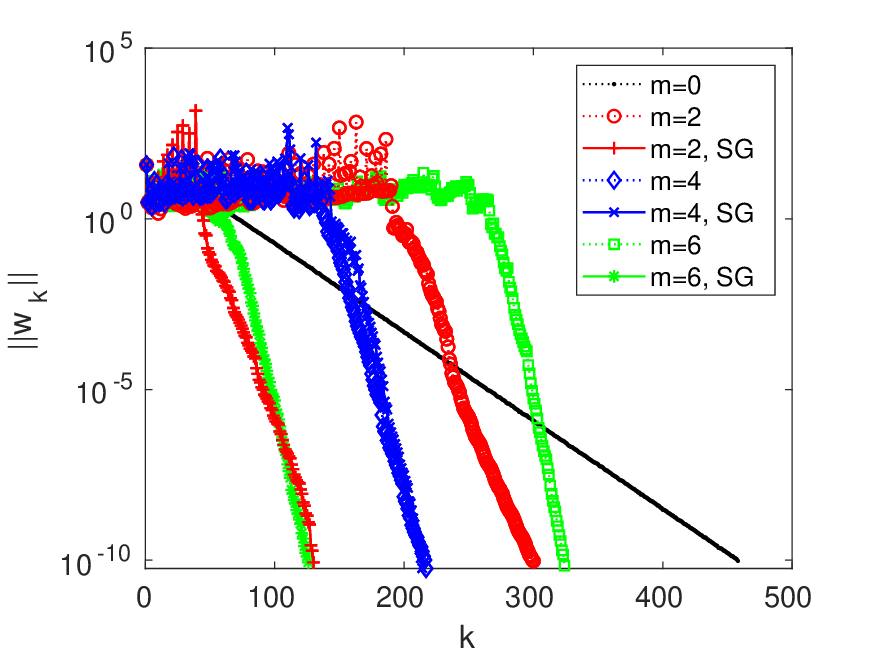}
~\includegraphics[trim = 0pt 0pt 10pt 10pt, clip=true, width = .49\textwidth]
{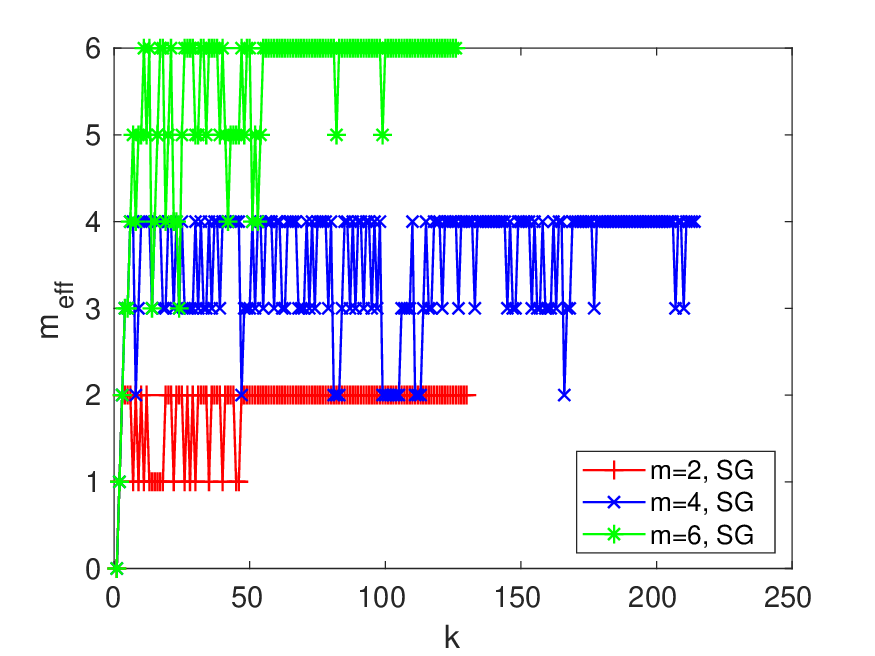}
\caption{Left: residual history for AA applied to \eqref{eqn:pLap1} with depths
$m = \{0,2,4,6\}$, with and without safeguarding (SG), which selects columns of $F_k$
to enforce \eqref{eqn:gmcond}, with threshold $c_s = 0.25$. Right: $m_{\text{eff}}$, 
the number of columns of $F_k$ selected at each stage in the safeguarded simulations.}
\label{fig:pLap246}
\end{center}
\end{figure}

\begin{figure}[H]
\begin{center}
\includegraphics[trim = 0pt 0pt 10pt 10pt, clip=true, width = .49\textwidth]
{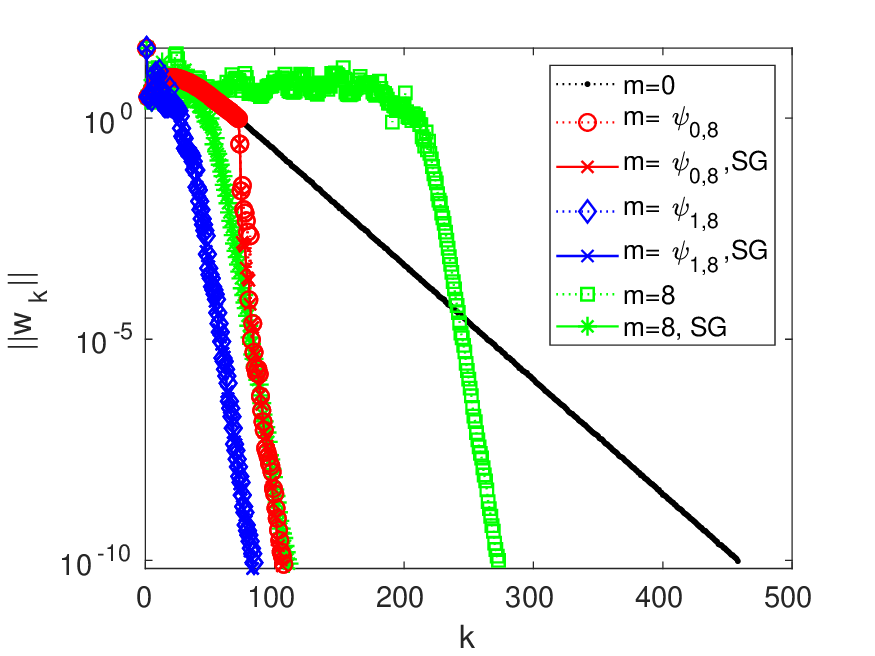}
~\includegraphics[trim = 0pt 0pt 10pt 10pt, clip=true, width = .49\textwidth]
{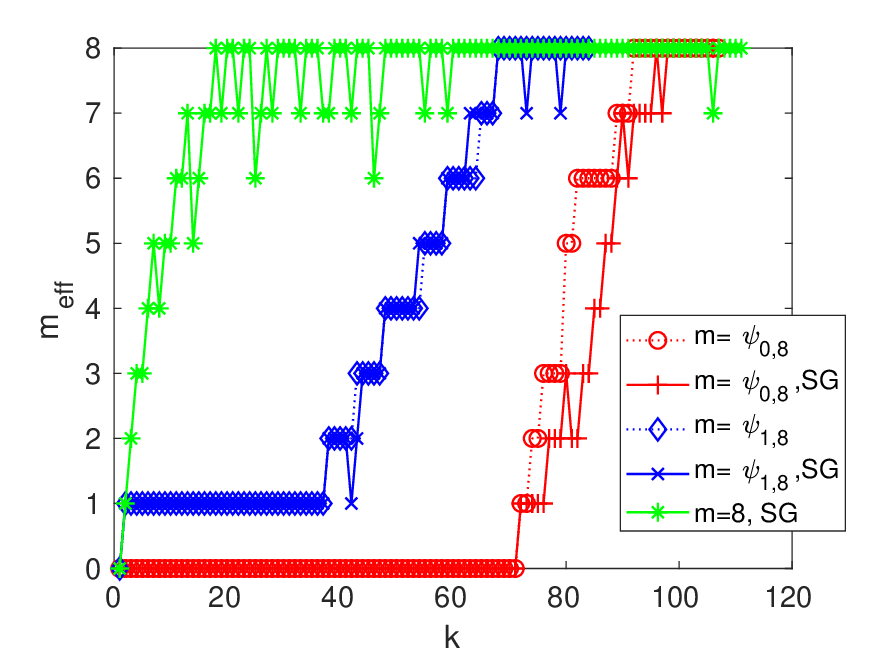}
\caption{Left: residual history for AA applied to \eqref{eqn:pLap1} with depths 
$m = \{0,8\}$, and dynamically selected depths $\psi_{0,8}$ and $\psi_{1,8}$ as 
given by \eqref{eqn:3phase}, with and without safeguarding (SG), 
which selects columns of $F_k$ to enforce \eqref{eqn:gmcond}, with threshold 
$c_s = 0.25$. Right: $m_{\text{eff}}$, 
the number of columns selected at each stage of the simulations.}
\label{fig:pLap018}
\end{center}
\end{figure}
Figure \ref{fig:pLap246} (left) shows residual histories for AA with constant depths
$m = \{0,2,4,6\}$, where $m= 0$ (the Picard iteration) is included for reference.
For $m= \{2,6\}$, the safeguarding procedure reduces the number of iterations from 
over 300 to just over 100.  This indicates that a near-linear dependence in the 
columns of $F_k$ produces an undesirable accelerated step toward the beginning of
those iterations. The safeguarding is seen to have little effect on any of the iterations
once they begin their rapid convergence. It also has little effect on the simulation with
$m=4$, indicating that near-linear dependence is not always an issue. The right
plot of Figure \ref{fig:pLap246} shows the number of columns of $F_k$ 
(denoted $m_{\text{eff}}$) selected for use by the safeguarding strategy. It is 
interesting to notice how for $m= \{2,6\}$, columns are deleted often toward the 
beginning of the simulations.  For $m=4$, columns are deleted throughout, but with 
little effect: the columns deleted caused little harm, but also
contributed little to the convergence.

Figure \ref{fig:pLap018} (left) shows residual histories using a constant depth of
$m=8$, together with the 3-phase dynamic depth selection strategy suggested in 
Subsection \ref{subsubsec:3phase}, and given by \eqref{eqn:3phase}, with depths
ranging from 0 to 8 using $\psi_{0,8}$ and from 1 to 8 with $\psi_{1,8}$.  
In this figure it is observed that the safeguarding has little effect on convergence
when it is used together with the dynamic depth selection; 
however, it has a substantial effect
on convergence for the constant depth $m=8$.  The right plot of Figure
\ref{fig:pLap018} shows the early intervention in removing certain columns from 
$F_k$ when $m_k = \min\{k-1,8\}$ leads to the fast convergence seen on the left.  
In contrast, while the safeguarding strategy does remove columns from $F_k$ periodically
using the dynamic $\psi_{1,8}$ and particularly for $\psi_{0,8}$, it leads to very little
change in the convergence histories in either case. 

This example shows that either dynamic depth selection or safeguarding can lead to 
improved convergence of AA.  The early stages of simulations, particularly if they
are started with poor initial iterates as is the case here, can be sensitive
to choice of depth without such interventions. Combining the two strategies did not
lead to a noticeable advantage or disadvantage.

\subsection{Nonlinear Helmholtz equation}
The following 1D nonlinear Helmholtz (NLH) equation, arises in nonlinear optics and 
describes the propagation of continuous-wave laser beams through transparent dielectrics.
Following the formulation from \cite{BFT07}, the system may be written as : 
Find $u:[0,10]\rightarrow \mathbb{C}$ satisfying
\begin{align*}
\frac{d^2 u}{dx^2} + k_0^2 \left( 1 + \epsilon(x) |u|^2 \right) u & = 0, \ \ \ 0<x<10, \\
\frac{du}{dx} + i k_0 u & = 2 i k_0,\ \ \ x=0,\\
\frac{du}{dx} - i k_0 u & = 0,\ \ \ x=10.
\end{align*}
Here, $\epsilon(x)$ is a given non-negative function of $x$ representing a material 
constant at each point in space, and $k_0$ is the linear wave number in the surrounding 
medium.  For simplicity, $\epsilon(x)=\epsilon$ is taken as a non-negative constant.

Even in 1D with constant $\eps(x) > 0$, 
this is a very challenging problem, especially for larger values of $\epsilon$ and 
$k_0$, each of which increases the effect of the cubic nonlinearity.
The system is discretized by applying a second order finite difference 
method (with uniform point spacing of $h=0.01$) to the iteration
\begin{align}
\frac{d^2 u_{j+1}}{dx^2} + k_0^2 \left( 1 + \epsilon |u_j|^2 \right) u_{j+1} & = 0, \ \ \ 0<x<10, \label{nlh1} \\
\frac{du_{j+1}}{dx} + i k_0 u_{j+1} & = 2 i k_0,\ \ \ x=0, \label{nlh2} \\
\frac{du_{j+1}}{dx} - i k_0 u_{j+1} & = 0,\ \ \ x=10. \label{nlh3}
\end{align}
This can be considered a fixed point iteration $u_{j+1}=g(u_j)$, 
with $g$ defined to be the solution operator of the (discretized) system 
\eqref{nlh1}-\eqref{nlh3}.  
Following \cite{BFT07}, $u_0=e^{ik_0 x}$ is used as the initial iterate.
\begin{figure}[ht!]
\begin{center}
\includegraphics[width = .45\textwidth, height=.40\textwidth,viewport=0 0 530 400, clip]{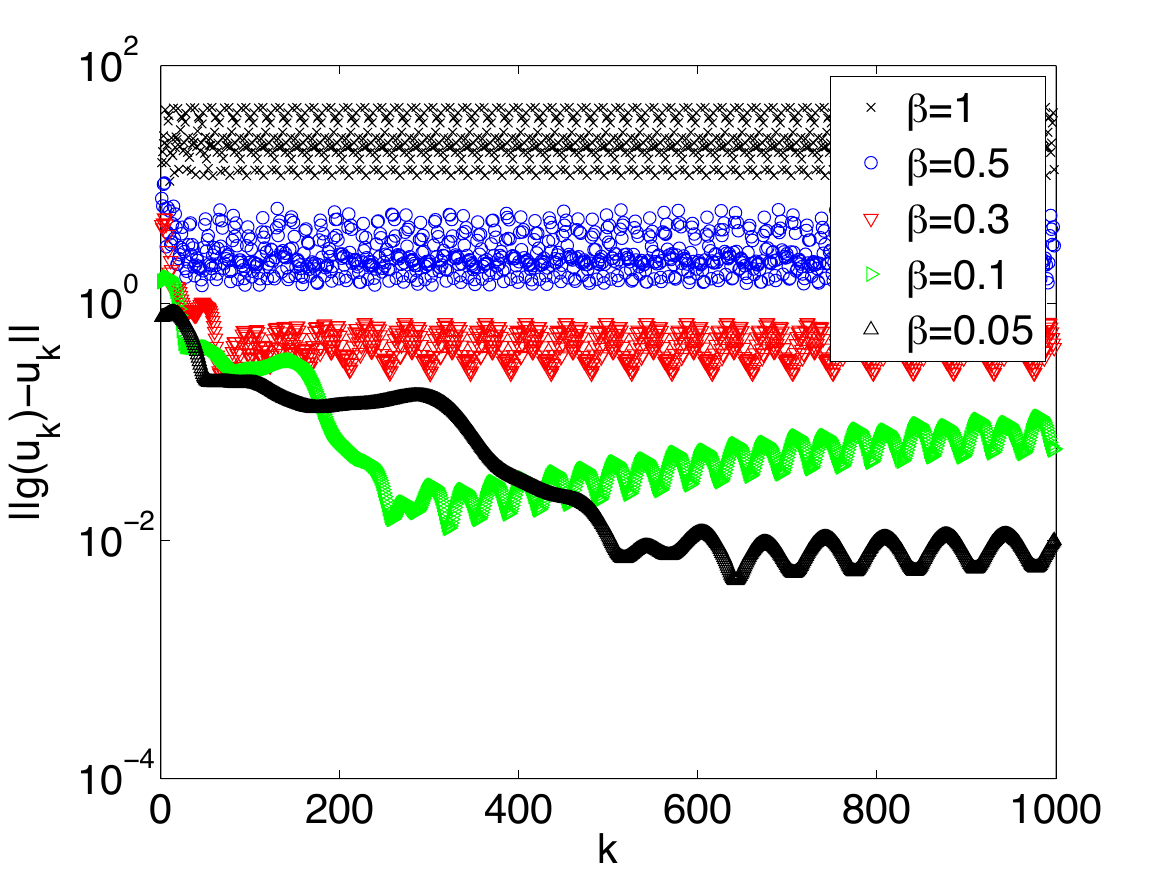}
\caption{Results of the NLH test with $\epsilon=0.22$ and $m=0$ demonstrating that
decreasing a fixed damping factor $\beta$ does not induce convergence of the 
fixed-point iteration.}
\label{eps22beta} 
\end{center}
\end{figure}

\begin{figure}[ht!]
\begin{center}
\includegraphics[width = .45\textwidth, height=.40\textwidth,viewport=0 0 530 400, clip]{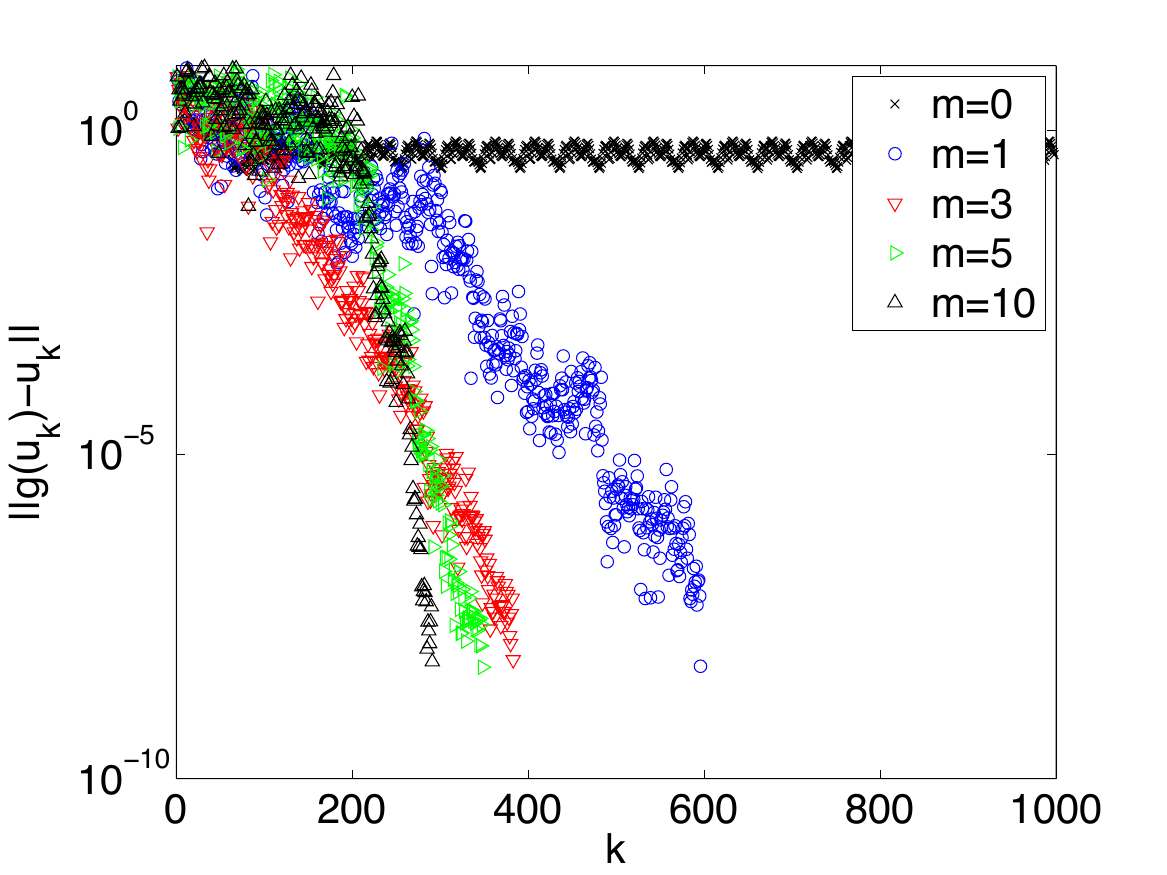}
\includegraphics[width = .45\textwidth, height=.40\textwidth,viewport=0 0 530 400, clip]{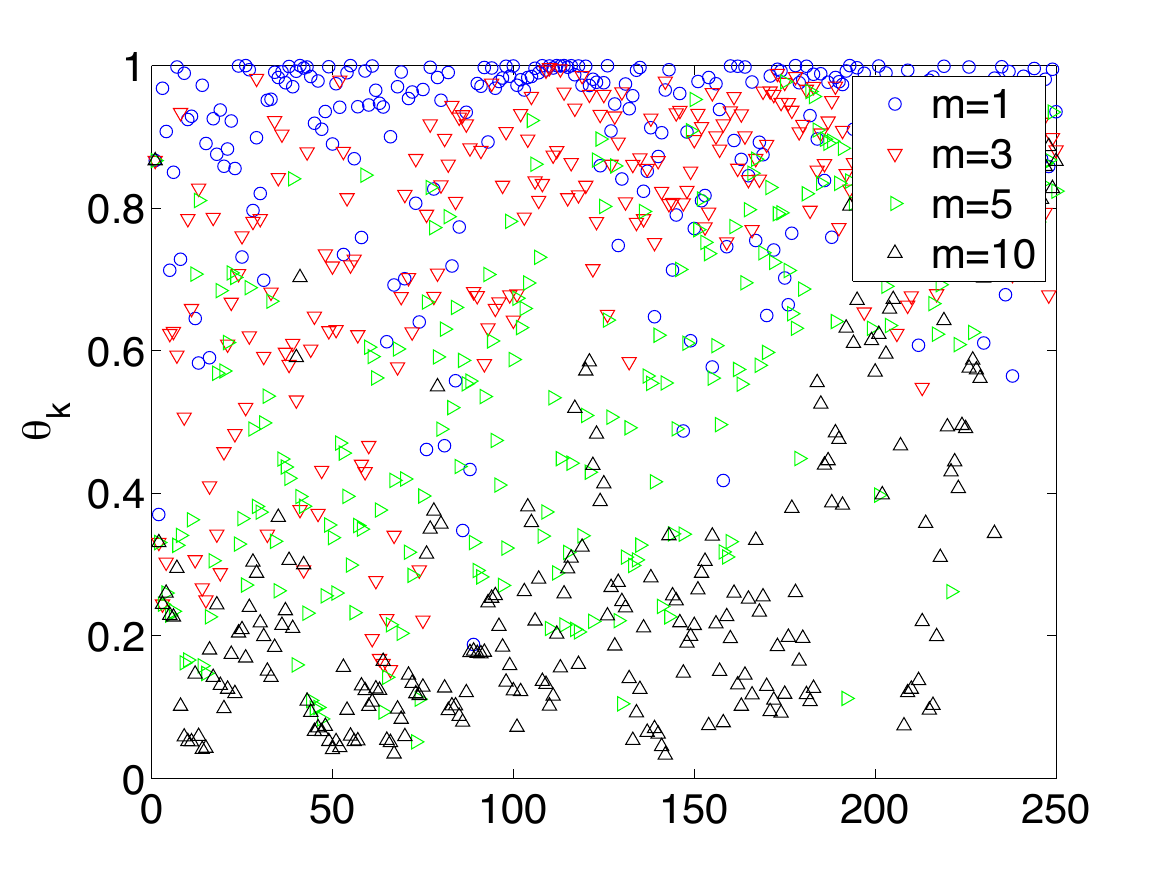}
\caption{Results of the NLH test with $\epsilon=0.22$, as convergence of the nonlinear residual (left) for $\beta_k=\beta=0.3$ and varying $m$, and $\theta_k$ for varying $m$ (right).}
\label{eps22conv} 
\end{center}
\end{figure}

This NLH test uses $\epsilon=0.22$, 
for which the fixed point iteration \eqref{nlh1}-\eqref{nlh3} does not converge.  
Figure \ref{eps22beta} shows the fixed-point iteration $(m=0)$ with varying levels of 
relaxation (damping); this illustrates that (uniform) relaxation alone is
not sufficient for convergence.  
In Figure \ref{eps22conv}, results of AA applied to the iteration using relaxation 
parameter $\beta_k=\beta=0.3$ are shown for $m = 1,3,5,10$, all of which converge.
The plot of $k$ {\em vs.} $\theta_k$ shows  a clear reduction in
gain factors $\theta_k$ as the depth $m$ increases.  
Comparing convergence histories for varying depths $m$, none of the depths tested
show monotonic decrease, particularly in the preasymptotic regime. Depth
$m=10$ which becomes nearly monotone in the asymptotic regime, has gain values 
generally less than $0.6$; whereas depth $m=1$ which is far from monotone has 
gain values that return to nearly one throughout the first 250 iterations shown 
in Figure \ref{eps22conv}, on the right.

\begin{figure}[ht!]
\begin{center}
\includegraphics[width = .45\textwidth, height=.40\textwidth,viewport=0 0 530 400, clip]{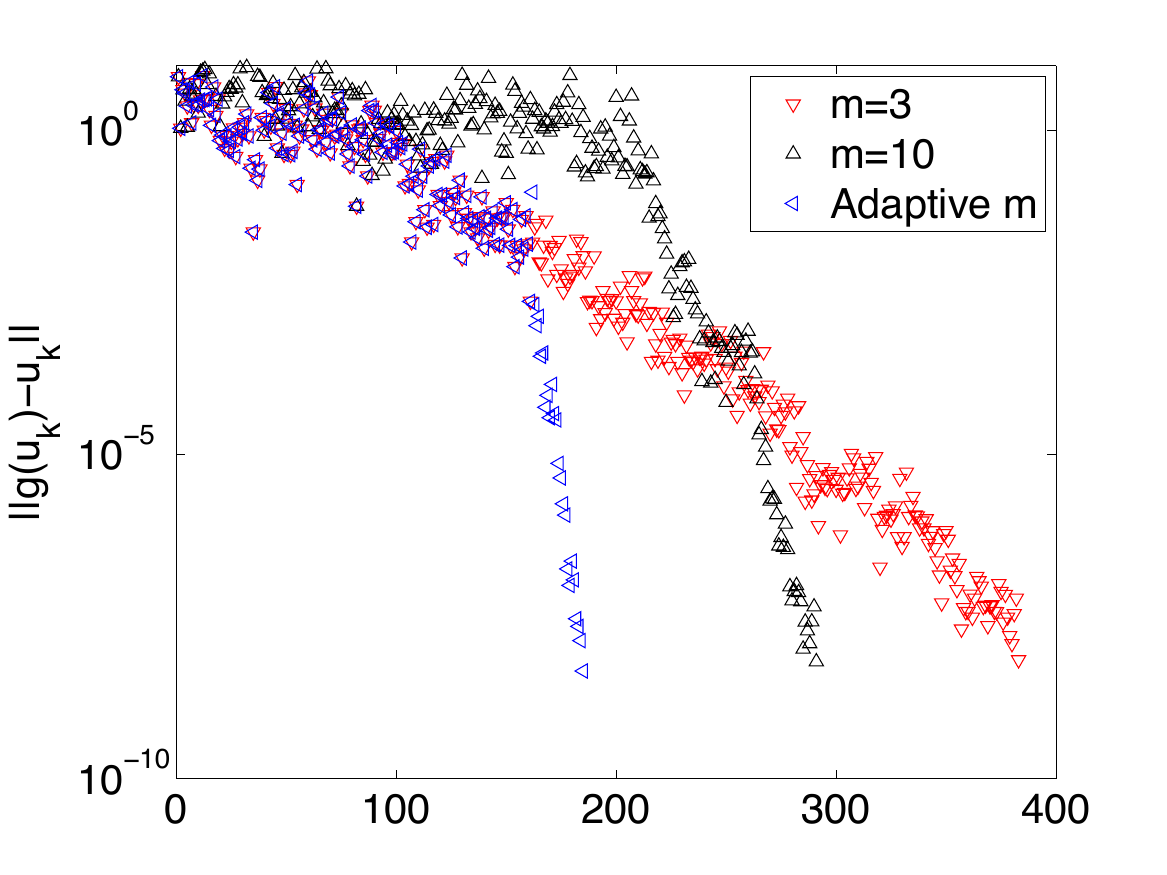}
\caption{Results of the NLH test with $\epsilon=0.22$, as convergence of the nonlinear residual (left) for $\beta_k=\beta=0.3$ and $m=3$, $m=10$, and an heuristic
strategy where $m=3$ at first but switches to $m=10$ once the nonlinear residual is sufficiently small.}
\label{eps22adapt} 
\end{center}
\end{figure}

The next results, shown in Figure \ref{eps22adapt}, 
use a heuristic strategy for updating $m$.
This strategy is based on the observation that
depth $m=3$ gives a faster initial decrease in the residual, and $m=10$ gives
the fastest eventual decrease. 
Here, depth $m_k$ is switched from $m_k = 3$ to $m_k=10$ on the condition
of a sufficiently small residual, where the tolerance is set at $0.005$.
The depth-switching approach yields substantially faster convergence than either 
constant-depth strategy.
This is again consistent with the theory, as larger higher order terms play a 
greater role earlier in the iteration history, and moreso at greater algorithmic 
depths. Once the higher order terms are sufficiently small, (attained through
a sequence of sufficiently small gain values), the decrease
in gain $\theta_k$ 
for greater depths $m$ yields better performance as the residual is small
enough to be dominated by the first order term even as the number and weight of the 
higher order terms increase.

\subsection{3D Steady Navier-Stokes equations}

The last example shown is for the 3D driven cavity benchmark test problem for 
the steady Navier-Stokes equations (NSE).  
The steady NSE are given in a domain $\Omega\subset \mathbb{R}^d$  (d=2,3) by
\begin{eqnarray}
u \cdot\nabla u + \nabla p - \nu\Delta u  & = & f, \label{ns1} \\
\nabla \cdot u & = & 0, \label{ns2} \\
u|_{\partial\Omega} & = & s, \label{ns3} 
\end{eqnarray}
where $\nu$ is the kinematic viscosity which is inversely proportional to the Reynolds 
number $Re:=\nu^{-1}$, $f$ is a forcing term, and $u$ and $p$ represent 
velocity and pressure.  The NSE are well-known to be more difficult to solve with
larger Reynolds number.

The 3D driven cavity is a widely studied benchmark problem for the NSE, and typically 
with $Re\le 1000$ (see \cite{WB02} and reference therein).  
For this problem, $\Omega=(0,1)^3$, and there is no forcing ($f=0$). For boundary 
conditions, $s=0$ is enforced on the bottom and sides, and on the top, 
$s=\langle 1,0,0\rangle^\top$, by which the driving force is provided by the 
moving lid.
Recently, higher $Re$ have been considered, but as a {\it time dependent flow}, in an 
attempt to find the first Hopf bifurcation where the flow becomes oscillatory and will 
not converge to a steady state \cite{CPHGG16,FG10}. 
This bifurcation appears to occur around $Re\approx 2000$.  
Here, the system \eqref{ns1}-\eqref{ns3} is solved 
by applying AA to the Picard iteration, given by \cite{GR86} as
\begin{eqnarray}
u_k \cdot\nabla u_{k+1} + \nabla p_{k+1} - \nu\Delta u_{k+1}  & = & f, \label{p1} \\
\nabla \cdot u_{k+1} & = & 0, \label{p2} \\
u_{k+1}|_{\partial\Omega} & = & s. \label{p3} 
\end{eqnarray}
The system above defines a fixed-point iteration with $u_{k+1}=g(u_k)$, 
where $g$ is the solution operator for a spatial discretization of \eqref{p1}-\eqref{p3}. 
The system is discretized using $(P_3,P_2^{disc})$ Scott-Vogelius finite elements on a 
barycenter refined tetrahedral mesh that provides 1.3 million total degrees of freedom.  
The tetrahedral mesh was created using a first a box mesh to subdivide
all axes using Chebyshev points (to be more refined near the boundary), then 
splitting each box into 6 tetrahedra, then splitting each tetrahedron  
with a barycenter refinement.  
The initial guess for each of the  NSE tests is $u_0=0$
(no continuation methods are applied).

\begin{figure}[ht!]
\begin{center}
\includegraphics[width = .98\textwidth, height=.3\textwidth,viewport=0 0 1100 350, clip]{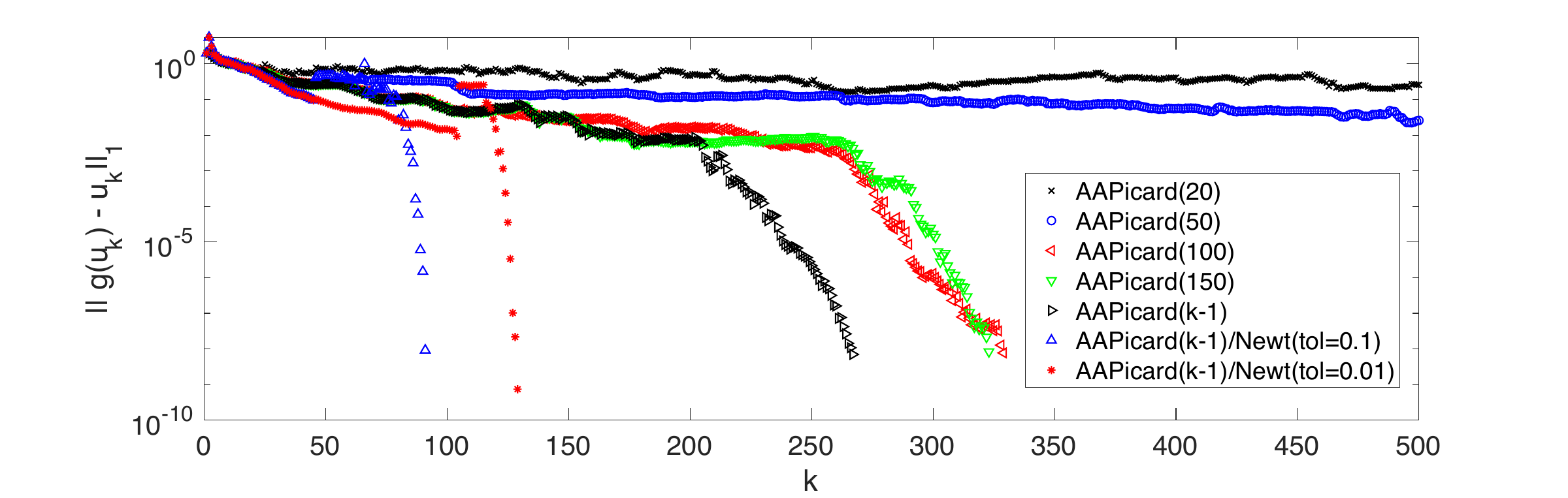}
\includegraphics[width = .98\textwidth, height=.3\textwidth,viewport=0 0 1100 350, clip]{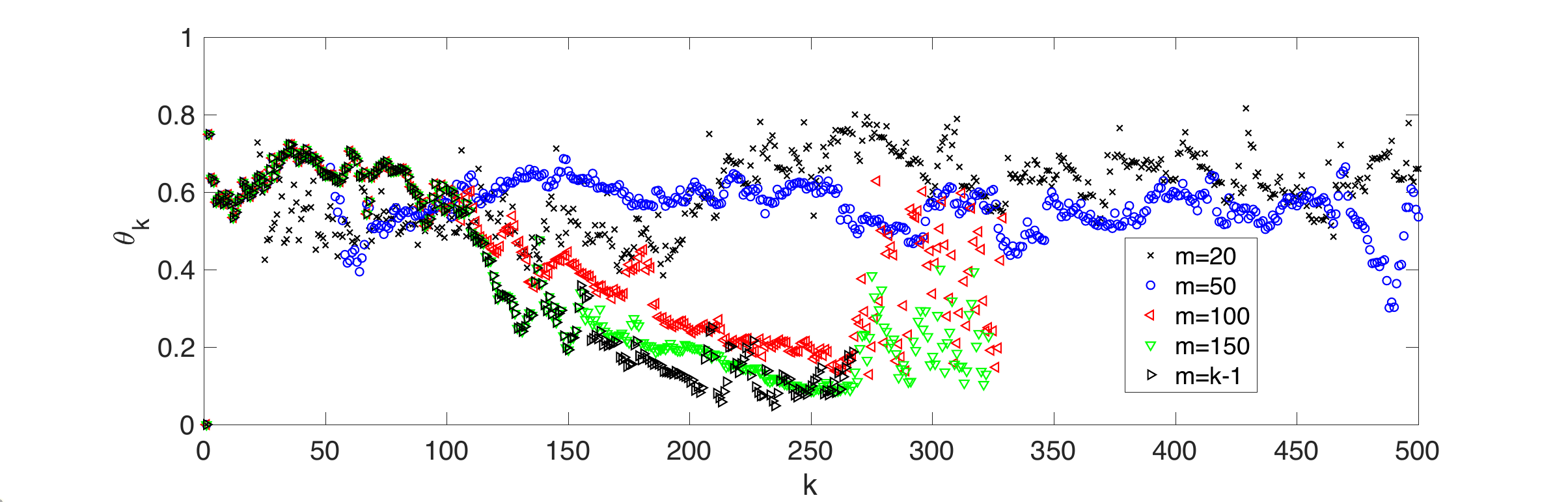}
\caption{Top: convergence of AAPicard with varying $m$ with and without a switch to 
Newton. Bottom: gain factors $\theta_k$ for varying $m$.}
\label{nseplots} 
\end{center}
\end{figure}

\begin{figure}[ht!]
\begin{center}
$Re=2500$\\
\includegraphics[width = \textwidth, height=.3\textwidth,viewport=50 0 1150 350, clip]{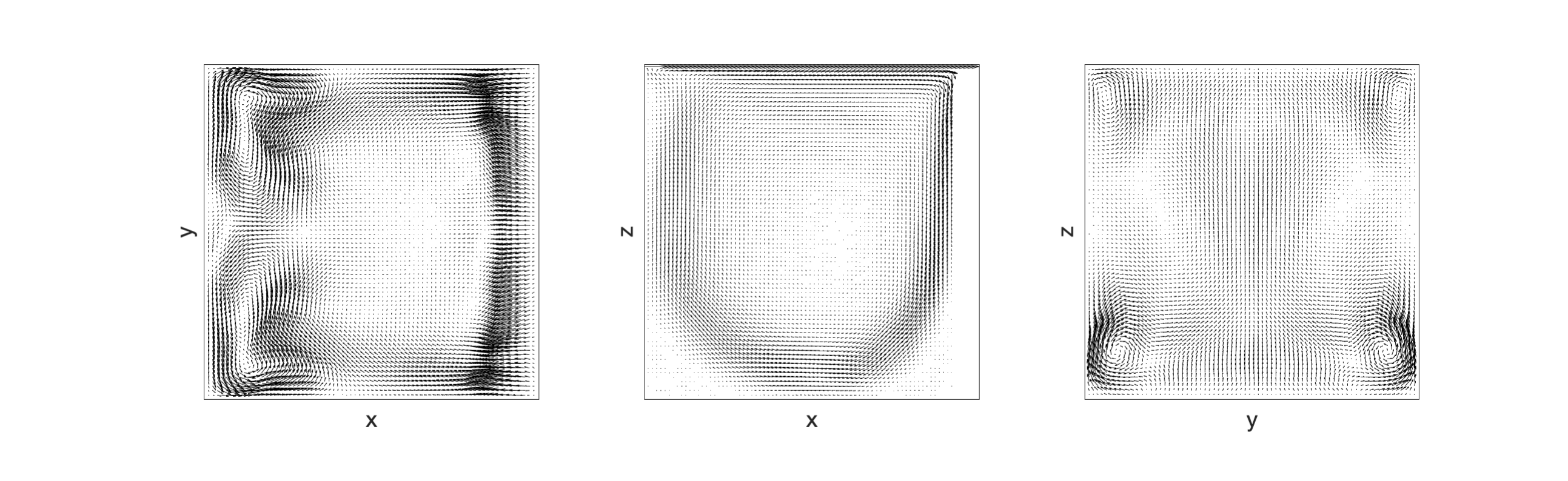}\\
$Re=3100$\\
\includegraphics[width = \textwidth, height=.3\textwidth,viewport=50 0 1150 350, clip]{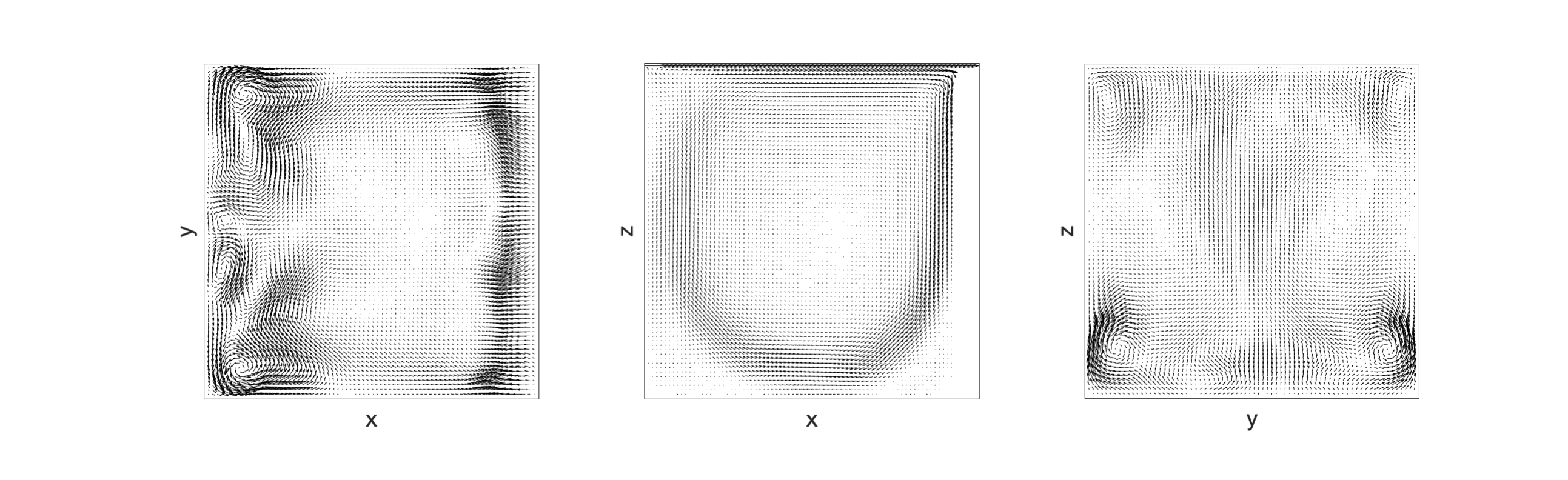}
\caption{Shown above are $Re=2500$ and $3100$ solutions, as midsliceplanes of the velocity fields.}
\label{nseplots2} 
\end{center}
\end{figure}

In the paper \cite{PRX18}, AA applied to \eqref{p1}-\eqref{p3} (referred to here as
AAPicard) was studied both theoretically and numerically.
Under a small data condition that implies the underlying fixed-point iteration is
contractive, it was shown that the method converges, 
and that the linear convergence rate is improved by AA. 
It is remarked however that the techniques used in that analysis and the 
coefficients in front of the higher order terms differ significantly from those shown 
here.

For the current test problem, as shown in \cite{PRX18}, with an initial guess of $u_0=0$,
the Picard method does not converge when $Re=400$.
Hence, for $Re\ge 400$, Picard iterations for steady solutions are not 
globally contractive. In fact, AAPicard with $m=1$ fails as well, although convergence
is attained with depths $m =2,3,4$ as demonstrated in \cite{PRX18}.
To show the effectiveness of AAPicard, considerably higher Reynolds numbers are considered
here: results are presented for $Re=2500$ and $Re=3100$, far beyond 
the range where the Picard iteration is contractive; and moreover, well past the 
first Hopf bifurcation  \cite{CPHGG16,FG10}.
Thus the method is converging to steady solutions in a time dependent regime, 
which from a mathematical point of view is interesting in itself. As discussed 
in \cite{ABHHMS06}, such solutions can serve as base-flow solutions in instability 
studies and flow control.

The $Re=2500$ tests show different choices of the depth $m$, including the largest 
possible ($m_k=k-1$), with no relaxation ($\beta_k=\beta=1$). 
Results are shown in Figure \ref{nseplots}.  
For $m\le 50$, convergence is not achieved (nor is it close to being achieved)
after 500 iterations. 
For $m=100,150$, and $m=k-1$, the method does converge. It appears that the stability 
of the NSE Picard iteration \cite{GR86} bounds the magnitude of any residual, 
and the improved analysis herein shows that higher order terms are all scaled by the 
latest residual, which together allows the method to benefit from the small gain factor
$\theta_k$ that comes from a greater algorithmic depth $m$ 
($m\ge 100$ creates gain factors $\theta_k$ that get to 0.25 and below).  
Notably, choosing $m$ as large as possible, $m_k=k-1$, 
gives the fastest convergence.

Finally, a combination of AAPicard with Newton ({\em cf}. \cite{FKR18}), was tested. 
The Newton iteration differs from the Picard in that the term 
$(u_{k+1}-u_k)\cdot\nabla u_k$ is added to the left side of \eqref{p1}.
Additionally, a line search was used in the Newton iterations.
The results shown used $m_k=k-1$ for the initial AAPicard iterations and switched to 
Newton once the nonlinear residual reached a sufficiently low tolerance.  
For a $H^1_0$-norm tolerance of $1$, the method failed to converge. For tolerances of 
$0.1$ and $0.01$, the method converged, and much faster than AAPicard on its own 
(see the top plot in Figure \ref{nseplots}).

With this technique, the solver attained convergence up to  $Re=3100$ 
(using AAPicard with $m_k=k-1$ and $\beta_k = \beta = 0.5$, 
up to a residual tolerance of $0.03$, then switching to Newton with a line search). 
With this method, 217 iterations were needed to converge to a tolerance of $10^{-8}$ in 
the $H^1_0$-norm.  With a continuation method that improves the initial guess, solutions at even 
higher $Re$ can be obtained.  Plots of the $Re=2500$ and $3100$ solutions are shown in 
Figure \ref{nseplots2} as midsliceplanes of the velocity fields.

\section{Conclusion}
The presented one-step analysis of Anderson acceleration sharpens the 
previously developed residual bounds for contractive operators and 
extends them to a class of potentially noncontractive operators which are 
important for the the approximation of solutions to nonlinear PDEs.  
The new analysis shows how the relative
scaling of the higher-order terms increases as the solution to the 
underlying optimization problem improves.  
Understanding the balance of the higher and lower order terms in the residual
expansion is instrumental in the design of robust and efficient algorithms
for challenging nonlinear problems.  
The current theory assumes that the latest difference between consecutive
residuals sufficiently changes the span of the previous differences, up to the 
given algorithmic depth. An efficient safeguarding strategy to ensure this assumption
holds is introduced and demonstrated, advancing the connection between theory 
and practice in a sense not accomplished with the usual assumption that the optimization 
coefficients are bounded.
Practical advantages based on the present advances in theory are
demonstrated in the numerical
section where Anderson acceleration is used to attain results for the nonlinear 
Helmholtz equation and 3D steady Navier-Stokes past the first Hopf bifurcation
which cannot be attained by the usual combinations of Picard iterations,
Newton iterations and relaxation techniques alone.

\section*{Acknowledgements}
SP is partially supported by NSF DMS 1852876 and 2011519.
LR is partially supported by NSF DMS 1522191 and 2011490.
The authors would like to thank the anonymous referees for suggesting additional 
clarification on the connection between the theory and the examples, and on the value 
of assuming \eqref{eqn:gmcond} instead the boundedness of the optimization coefficients.

\bibliographystyle{abbrv}
\bibliography{graddiv}

\appendix
\section*{Appendix}
The proof of the technical Lemma \ref{lem:invR} follows.

\begin{proof}
The proof follows by induction on the submatrix formed by the first $p$ rows and 
columns of $R$, then by induction indexing up the entries of the right-most column.
Let $R_p = R(1:p,1:p)$, the upper-left $p \times p$ block of $\hat R$, 
with inverse $S_p$.

The off-diagonal entries $r_{ij}$ of $\hat R$ are given by 
$r_{ij} = (q_i,a_j) = \nr{a_j}\cos(q_i,a_j)$, 
and by Proposition \ref{prop:diag-entries} the diagonal entries are given 
by $r_{ii} = \nr{a_i}|\sin(a_i, \cA_{i-1})|$, following the convention that
the columns of $\hat Q$ are chosen so the $r_{ii}$ are positive.

For the trivial case of $p=1$, $R_1 = r_{11}$, and
$s_{11} = 1/r_{11}= 1/\nr{a_1}$.
By Proposition \ref{prop:tridiag-inv}, to compute the inverse of $R_2$
it remains to compute $s_{22}$ and $s_{12}$. 
It is useful here to state the inversion formula for entries of the right-most
column (index $p$) as
\begin{align}\label{eqn:inv001}
s_{pp} = \f{1}{r_{pp}}, ~\text{ and }~
s_{kp} = -\f{1}{r_{kk}} \sum_{j = 1}^{p-k}r_{k,k+j}s_{k+j,p}
       = -\f{1}{r_{kk}} \sum_{j = 1}^{p-k}\nr{a_{k+j}}\cos(q_k,a_{k+j})s_{k+j,p},
~ k < p.
\end{align}
For $p=2$, the inversion formula \eqref{eqn:inv001} and expression
\eqref{eqn:diag-entries} for the diagonal entries yield
$s_{22} = 1/r_{22}= 1/(\nr{a_2}|\sin(a_2,q_1)|)$.
Then by the hypotheses of the lemma,
$s_{22} \le 1/(\nr{a_2}c_s)$. 
Using \eqref{eqn:inv001}, the off-diagonal entry then satisfies
$s_{12} = -\nr{a_2} \cos(q_1,a_2)s_{22}/r_{11}$, yielding 
$|s_{12}| \le c_t/(\nr{a_1}c_s)$.  Hence for $p=2$ the result holds.  Continue
by induction on $p$, assuming the result holds for $q = 1, \ldots, p-1$. Then for
$q=p$,
\[
R_p = \left(\begin{array}{cccccc}R_{p-1} & r_{1p}\\ & \vdots \\0 & r_{pp} 
\end{array}\right).
\]
By \eqref{eqn:inv001}, Proposition \ref{prop:diag-entries} and the 
hypotheses of the lemma, $\nr{s_{pp}} \le 1/(\nr{a_p}c_s)$. 

Similarly by \eqref{eqn:inv001}, $\nr{s_{p-1,p}} \le c_t/(\nr{a_{p-1}}c_s^2).$
This satisfies the 
base step on the inner induction, up row $p$ of $S_p$. Assuming the bound 
of \eqref{eqn:invR} for $s_{ip}$ holds for $i = p-1$ down to $i=k+1$, it suffices
to show the result for $i=k$. By \eqref{eqn:inv001} and the inductive hypothesis,
\begin{align*}
|s_{kp}| &= 
\left| \f{1}{r_{kk}} \sum_{j = 1}^{p-k}\nr{a_{k+j}}\cos(q_k,a_{k+j})s_{k+j,p} \right|
\le 
\f{1}{r_{kk}} \left( \sum_{j = 1}^{p-k-1} 
\f{c_t^2(c_t+c_s)^{p - (k+j) -1}}{c_s^{p-(k+j)+1}} 
+ \f {c_t} {c_s} \right).
\end{align*}
Setting $n = p-k$ 
\begin{align}\label{eqn:inv003}
|s_{kp}| &\le 
\f{1}{r_{kk}} \left( \sum_{j = 1}^{n-1} \f{c_t^2(c_t+c_s)^{n-j -1}}{c_s^{n-j+1}} 
+ \f {c_t} {c_s} \right)
= \f{c_t}{c_s^n r_{kk}} 
\left( \sum_{j = 1}^{n-1} c_t(1+c_s)^{n-j -1} c_s^{j-1} + c_s^{n-1}
   \right).
\end{align}
Rearranging the terms in the sum shows
\begin{align}\label{eqn:inv004}
&\sum_{j = 1}^{n-1} c_t(c_t+c_s)^{n-j -1} c_s^{j-1} + c_s^{n-1} 
\nonumber \\
& = \sum_{j = 1}^{n-2} c_t(c_t+c_s)^{n-j -1} c_s^{j-1} + (c_t+ c_s)c_s^{n-2}
\nonumber \\ & = 
\sum_{j = 1}^{n-3} c_t(c_t+c_s)^{n-j -1} c_s^{j-1} + (c_t+c_s)^2c_s^{n-3}
\nonumber \\ & \vdots \nonumber \\ & = 
c_t(c_t+c_s)^{n-2} + c_t(c_t+c_s)^{n-3}c_s + (c_t+c_s)^{n-3}c_s^2
\nonumber \\ & = c_t(c_t+c_s)^{n-2} + (c_t+c_s)^{n-2} c_s 
\nonumber \\ &= (c_t+c_s)^{n-1}.
\end{align}
Applying \eqref{eqn:inv004} and \eqref{eqn:diag-entries} to \eqref{eqn:inv003}
allows
\begin{align*}
|s_{1p}| \le \f{c_t(c_t+c_s)^{p-2}}{\nr{a_1}c_s^{p-1}}, ~\text{ and }~
|s_{kp}| \le \f{c_t(c_t+c_s)^{p-k-1}}{\nr{a_k}c_s^{p-k+1}}, ~ k = 2, \ldots, p-1,
\end{align*}
which completes the inductive step on $k$ and hence on $p$, and establishes the result.
\end{proof}

\vfill

\end{document}